\documentclass[3p]{elsarticle}

\usepackage{bm}
\usepackage[parfill]{parskip}
\usepackage[utf8]{inputenc}
\usepackage{url}
\usepackage{amsmath,amssymb,amsfonts,amsthm}
\usepackage{color}
\usepackage{bm}
\usepackage{graphicx}
\usepackage{subcaption}
\usepackage{lipsum}
\usepackage{amsfonts}
\usepackage{graphicx}
\usepackage{epstopdf}
\usepackage{array}
\usepackage[hidelinks]{hyperref}

\usepackage{amsmath}
\usepackage{algorithm}
\usepackage[noend]{algpseudocode}

\ifpdf
  \DeclareGraphicsExtensions{.eps,.pdf,.png,.jpg}
\else
  \DeclareGraphicsExtensions{.eps}
\fi

\newtheorem{theorem}{Theorem}

\DeclareMathOperator*{\argmin}{arg\,min}

\makeatletter
\newsavebox\myboxA
\newsavebox\myboxB
\newlength\mylenA

\newcommand*\xoverline[2][0.75]{%
    \sbox{\myboxA}{$\m@th#2$}%
    \setbox\myboxB\null
    \ht\myboxB=\ht\myboxA%
    \dp\myboxB=\dp\myboxA%
    \wd\myboxB=#1\wd\myboxA
    \sbox\myboxB{$\m@th\overline{\copy\myboxB}$}
    \setlength\mylenA{\the\wd\myboxA}
    \addtolength\mylenA{-\the\wd\myboxB}%
    \ifdim\wd\myboxB<\wd\myboxA%
       \rlap{\hskip 0.5\mylenA\usebox\myboxB}{\usebox\myboxA}%
    \else
        \hskip -0.5\mylenA\rlap{\usebox\myboxA}{\hskip 0.5\mylenA\usebox\myboxB}%
    \fi}
\makeatother

\journal{arXiv.org}

\newcommand{\TheTitle}{Intrusive acceleration strategies for Uncertainty Quantification for hyperbolic systems of conservation laws} 

\date{\today}

\usepackage{amsopn}

\journal{SIAM Journal on Scientific Computing}

\begin{document}

\begin{frontmatter}

\title{\TheTitle}

\author[adressJonas]{Jonas Kusch}
\author[adressJannick]{Jannick Wolters}
\author[adressMartin]{Martin Frank}

\address[adressJonas]{Karlsruhe Institute of Technology, Karlsruhe,
    jonas.kusch@kit.edu}
\address[adressJannick]{Karlsruhe Institute of Technology, Karlsruhe, jannick.wolters@kit.edu}
\address[adressMartin]{Karlsruhe Institute of Technology, Karlsruhe, martin.frank@kit.edu}

\begin{abstract}
Methods for quantifying the effects of uncertainties in hyperbolic problems can be divided into intrusive and non-intrusive techniques. Non-intrusive methods allow the usage of a given deterministic solver in a black-box manner, while being embarrassingly parallel. On the other hand, intrusive modifications allow for certain acceleration techniques. Moreover, intrusive methods are expected to reach a given accuracy with a smaller number of unknowns compared to non-intrusive techniques. This effect is amplified in settings with high dimensional uncertainty. A downside of intrusive methods is the need to guarantee hyperbolicity of the resulting moment system. In contrast to stochastic-Galerkin (SG), the Intrusive Polynomial Moment (IPM) method is able to maintain hyperbolicity at the cost of solving an optimization problem in every spatial cell and every time step.\\
In this work, we propose several acceleration techniques for intrusive methods and study their advantages and shortcomings compared to the non-intrusive Stochastic Collocation method. When solving steady problems with IPM, the numerical costs arising from repeatedly solving the IPM optimization problem can be reduced by using concepts from PDE-constrained optimization. Integrating the iteration from the numerical treatment of the optimization problem into the moment update reduces numerical costs, while preserving local convergence. Additionally, we propose an adaptive implementation and efficient parallelization strategy of the IPM method. The effectiveness of the proposed adaptations is demonstrated for multi-dimensional uncertainties in fluid dynamics applications, resulting in the observation of requiring a smaller number of unknowns to achieve a given accuracy when using intrusive methods. Furthermore, using the proposed acceleration techniques, our implementation reaches a given accuracy faster than Stochastic Collocation.
\end{abstract}


\end{frontmatter}

\section{Introduction}
Hyperbolic equations play an important role in various research areas such as fluid dynamics or plasma physics. Efficient numerical methods combined with robust implementations are widely available for these problems, however they do not account for uncertainties which can arise in measurement data or modeling assumptions. Including the effects of uncertainties in differential equations has become an important topic in the last decades. 

One strategy to represent the solution's dependence on uncertainties is to use a polynomial chaos (PC) expansion \cite{wiener1938homogeneous,xiu2002wiener}, i.e. the uncertainty space is spanned by polynomial basis functions. The remaining task is then to determine adequate expansion coefficients, often called moments or PC coefficients. Numerical methods for approximating these coefficients can be divided into intrusive and non-intrusive techniques. A popular non-intrusive method is the Stochastic Collocation (SC) method, see e.g. \cite{xiu2005high,babuvska2007stochastic,loeven2008probabilistic}, which computes the moments with the help of a numerical quadrature rule. Commonly, SC uses sparse grids, since they possess a reduced number of collocation points for multi-dimensional problems. Since the solution at a fixed quadrature point can be computed by a standard deterministic solver, the SC method does not require a significant implementation effort. Furthermore, SC is embarrassingly parallel, since the required computations decouple and the workload can easily be distributed across several processors. 

The main idea of intrusive methods is to derive a system of equations describing the time evolution of the moments which can then be solved with a deterministic numerical scheme. A popular approach to describe the moment system is the stochastic-Galerkin (SG) method \cite{ghanem2003stochastic}, which chooses a polynomial basis ansatz of the solution and performs a Galerkin projection to derive a closed system of equations. One significant drawback of SG is, that its moment system is not necessarily hyperbolic \cite{poette2009uncertainty}. A generalization of stochastic-Galerkin, which ensures hyperbolicity is the Intrusive Polynomial Moment (IPM) method \cite{poette2009uncertainty}. Instead of performing the PC expansion on the solution, the IPM method represents the entropy variables with polynomials. Besides yielding a hyperbolic moment system, the IPM method has several advantages: Choosing a quadratic entropy yields the stochastic-Galerkin moment system, i.e. IPM generalizes different intrusive methods. Furthermore, at least for scalar problems, IPM is significantly less oscillatory compared to SG \cite{kusch2017maximum}. Also, as discussed in \cite{poette2009uncertainty}, when choosing a physically correct entropy of the deterministic problem, the IPM solution dissipates the expectation value of the entropy, i.e. the IPM method yields a physically correct entropy solution. Unfortunately, the desirable properties of IPM come along with significantly increased numerical costs, since IPM requires the repeated computation of the entropic expansion coefficients from the moment vector, which involves solving a convex optimization problem. However, IPM and minimal entropy methods in general are well suited for modern HPC architectures \cite{garrett2015optimization}.

When studying hyperbolic equations, the moment approximations of various methods such as stochastic-Galerkin \cite{le2004uncertainty,kusch2018filtered}, IPM \cite{kusch2018filtered,poette2019contribution} and Stochastic Collocation \cite{barth2013non,dwight2013adaptive} tend to show incorrect discontinuities in certain regions of the physical space. These non-physical structures dissolve when the number of basis functions is increased \cite{pettersson2009numerical,offner2017stability} or when artificial diffusion is added through either the spatial numerical method \cite{offner2017stability} or by filters \cite{kusch2018filtered}. Also, a multi-element approach which divides the uncertain domain into cells and uses piece-wise polynomial basis functions to represent the solution has proven to mitigate non-physical discontinuities \cite{wan2006multi,durrwachter2018hyperbolicity}. Non-intrusive Monte-Carlo methods \cite{mishra2012multi,mishra2012sparse,mishra2016numerical}, which randomly sample input uncertainties to compute quantities of interest are robust, but suffer from a slow rate of convergence while again lacking the ability to use adaptivity to its full extent. Discontinuous structures commonly arise on a small portion of the space-time domain. Therefore, intrusive methods seem to be an adequate choice since they are well suited for adaptive strategies. By locally increasing the polynomial order \cite{tryoen2012adaptive,kroker2012finite,giesselmann2017posteriori} or adding artificial viscosity \cite{kusch2018filtered} at certain spatial positions and time steps in which complex structures such as discontinuities occur, a given accuracy can be reached with significantly reduced numerical costs. In addition to that, the number of moments needed to obtain a certain order with intrusive methods is smaller than the number of quadrature points for SC. Furthermore, collocation methods typically require a higher number of unknowns than intrusive methods to reach a given accuracy \cite{xiu2009fast,alekseev2011estimation}. Therefore, one aim should be to accelerate intrusive methods, since they can potentially outperform non-intrusive methods in complex and high-dimensional settings. \\

In this paper, we propose acceleration techniques for intrusive methods and compare them against Stochastic Collocation. For steady and unsteady problems, we use adaptivity, for which intrusive methods provide a convenient framework: 
\begin{itemize}
\item Since complex structures in the uncertain domain tend to arise in small portions of the spatial mesh, our aim is to locally increase the accuracy of the stochastic discretization in regions that show a complex structure in the random domain, while choosing a low order method in the remainder. Such an adaptive treatment cannot be realized with non--intrusive methods, since one needs to break up the black-box approach. To guarantee an efficient implementation, we propose an adaptive discretization strategy for IPM.
\end{itemize}
A steady problem provides different opportunities to take advantage of features of intrusive methods: 
\begin{itemize}
\item When using adaptivity, one can perform a large number of iterations to the steady state solution on a low number of moments and increase the maximal truncation order when the distance to the steady state has reached a specified barrier. Consequently, a large number of iterations will be performed by a cheap, low order method, i.e. we can reduce numerical costs. 
\item Perform an inexact map from the moments to the entropic expansion coefficients for IPM: Since the moments during the iteration process are inaccurate, i.e. they are not the correct steady state solution, we propose to not fully converge the dual iteration, which solves the IPM optimization problem. Consequently, the entropic expansion coefficients and the moments are converged simultaneously to their steady state, which is similar to the idea of One-Shot optimization in shape optimization \cite{hazra2005aerodynamic}.
\end{itemize}

The effectiveness of these acceleration ideas are tested by comparing results with Stochastic Collocation for the NACA test case \cite{jacobs1933characteristics} with uncertainties as well as a bent shock tube problem. Our numerical studies show the following main results:
\begin{itemize}
\item In our test cases, the need to solve an optimization problem when using the IPM method leads to a significantly higher run time than SC and SG. However when using the discussed acceleration techniques, IPM requires the shortest time to reach a given accuracy.
\item Comparing SG with IPM, one observes that for the same number of unknowns, SG yields more accurate expectation values, whereas IPM shows improved variance approximations.
\item By studying aliasing effects, we show that SC requires a higher number of unknowns than intrusive methods (even for a one-dimensional uncertainty) to reach the same accuracy level.
\item Using sparse grids for the IPM discretization when the space of uncertainty is multi-dimensional, the number of quadrature points needed to guarantee sufficient regularity of the Hessian matrix is significantly increased.
\end{itemize}
The IPM and SG calculations use a semi-intrusive numerical method, meaning that the discretization allows recycling a given deterministic code to generate the IPM solver. While facilitating the task of implementing general intrusive methods, this framework reduces the number of operations required to compute numerical fluxes. Also, it provides the ability to base the intrusive method on the same deterministic solver as used in the implementation of a black-box fashion Stochastic Collocation code, which allows for, what we believe, a fair comparison between intrusive and non-intrusive methods. The code is publicly available to allow reproducibility \cite{uqcreator}.

The paper is structured as follows: After the introduction, we present the discussed methods in more detail in section~\ref{sec:background}. The numerical discretization as well as the implementation and structure of the semi-intrusive method is introduced in section \ref{sec:framework}. In section~\ref{sec:OneShotIPM}, we discuss the idea of not converging the dual iteration. Section~\ref{sec:adaptivity} extends the presented numerical framework to an algorithm making use of adaptivity. Implementation and parallelization details are given in section~\ref{sec:parallel}. A comparison of results computed with the presented methods is then given in section \ref{sec:results}, followed by a summary and outlook in section \ref{sec:summary_outlook}.

\section{Background}
\label{sec:background}
In the following, we briefly introduce the notation and methods used in this work. A general hyperbolic set of equations with random initial data can be written as
\begin{subequations}\label{eq:hyperbolicProblem}
\begin{align}
\label{eq:fulleq}\partial_t \bm{u}(t,\bm{x},\bm{\xi}) + \nabla&\cdot\bm{f}(\bm{u}(t,\bm{x},\bm{\xi})) = \bm{0} \enskip \text{ in } D, \\ \label{eq:ic}
\bm{u}(t=0,\bm{x},&\bm{\xi}) = \bm{u}_{\text{IC}}(\bm{x},\bm{\xi}),
\end{align}
\end{subequations}
where the solution $\bm u\in\mathbb{R}^m$ depends on time $t\in\mathbb{R}^+$, spatial position $\bm{x}\in D\subseteq \mathbb{R}^d$ as well as a vector of independent random variables $\bm{\xi}\in\Theta\subseteq\mathbb{R}^p$ with given probability density functions $f_{\Xi,i}(\xi_i)$ for $i = 1,\cdots,p$. Hence, the probability density function of $\bm{\xi}$ is $f_{\Xi}(\bm\xi):=\prod_{i=1}^p f_{\Xi,i}(\xi_i)$. The physical flux is given by $\bm{f}:\mathbb{R}^m\to\mathbb{R}^{d\times m}$. To simplify notation, we assume that only the initial condition is random, i.e. $\bm{\xi}$ enters through the definition of $\bm{u}_{IC}$. Equations \eqref{eq:hyperbolicProblem} are usually supplemented with boundary conditions, which we will specify later for the individual problems.

Due to the randomness of the solution, one is interested in determining the expectation value or the variance, i.e.
\begin{align*}
\text{E}[\bm{u}] = \langle \bm{u} \rangle,\qquad \text{Var}[\bm{u}] = \langle \left( \bm{u}-\text{E}[\bm{u}]\right)^2\rangle,
\end{align*}
where we use the bracket operator $\langle \cdot \rangle := \int_{\Theta} \cdot f_{\Xi}(\bm\xi)d\xi_1 \cdots d\xi_p$. To approximate quantities of interest (such as expectation value, variance or higher order moments), the solution is spanned with a set of polynomial basis functions $\varphi_{i}:\Theta\to\mathbb{R}$ such that for the multi-index $i = (i_1,\cdots,i_p)$ we have $|i| \leq M$. Note that this yields
\begin{align}\label{eq:numberBasisFcts}
N:=\begin{pmatrix}
M+p \\ p
\end{pmatrix}
\end{align}
basis functions when defining $|i|:=\sum_{n = 1}^p |i_n|$. Commonly, these functions are chosen to be orthonormal polynomials \cite{wiener1938homogeneous} with respect to the probability density function, i.e. $\langle \varphi_i \varphi_j \rangle =\prod_{n=1}^p\delta_{i_nj_n}$. The generalized polynomial chaos (gPC) expansion \cite{xiu2002wiener} approximates the solution by
\begin{align}\label{eq:SGClosure}
\mathcal{U}(\bm{\hat u};\bm\xi):= \sum_{|i|\leq M} \bm{\hat{u}}_i\varphi_i(\bm{\xi}) = \bm{\hat{u}}^T\bm{\varphi}(\bm\xi),
\end{align}
where the deterministic expansion coefficients $\bm{\hat{u}}_i\in\mathbb{R}^m$ are called moments. To allow a more compact notation, we collect the $N$ moments for which $\vert i \vert \leq M$ holds in the moment matrix $\bm{\hat{u}}:=(\bm{\hat{u}}_i)_{|i|\leq M}\in\mathbb{R}^{N\times m}$ and the corresponding basis functions in $\bm{\varphi}:=(\varphi_i)_{|i|\leq M}\in\mathbb{R}^{N}$. In the following, the dependency of $\mathcal{U}$ on $\bm \xi$ will occasionally be omitted for sake of readability. The solution ansatz \eqref{eq:SGClosure} is $L^2$-optimal, if the moments are chosen to be the Fourier coefficients $\bm{\hat u}_i \equiv \langle \bm{u}\varphi_i \rangle\in\mathbb{R}^m$. One can also use the ansatz \eqref{eq:SGClosure} to compute the quantities of interest as
\begin{align*}
\text{E}[\mathcal{U}(\bm{\hat u})] = \bm{\hat u}_0,\quad \text{Var}[\mathcal{U}(\bm{\hat u})] = \text{E}[\mathcal{U}(\bm{\hat u})^2] - \text{E}[\mathcal{U}(\bm{\hat u})]^2 = \left(\sum_{i = 1}^N \hat{u}_{\ell i}^2\right)_{\ell = 1,\cdots,m}.
\end{align*}

The core idea of the Stochastic Collocation method is to compute the moments in the gPC expansion with a quadrature rule. Given a set of $Q$ quadrature weights $w_k$ and quadrature points $\bm{\xi}_k$, the moments are approximated by
\begin{align*}
\bm{\hat u}_i = \langle \bm{u}\varphi_i \rangle \approx \sum_{k = 1}^{Q}w_k \bm{u}({t,\bm{x},\bm{\xi}_k})\varphi_i(\bm{\xi}_k)f_{\Xi}(\bm{\xi}_k).
\end{align*} 
Hence, the moments can be computed by running a given deterministic solver for the original problem at each quadrature point. To reduce numerical costs in multi-dimensional settings, SC commonly uses sparse grids as quadrature rule: While tensorized quadrature sets require $O(M^p)$ quadrature points to integrate polynomials of maximal degree $M$ exactly, sparse grids are designed to integrate polynomials of total degree $M$, for which they only require $O(M(\log_2(M)^{p-1}))$ quadrature points, see e.g. \cite{trefethen2017cubature}.

Intrusive methods derive a system which directly describes the time evolution of the moments: Plugging the solution ansatz \eqref{eq:SGClosure} into the set of equations \eqref{eq:hyperbolicProblem} and projecting the resulting residual to zero yields the stochastic-Galerkin moment system
\begin{subequations}\label{eq:SGMomentSystem}
\begin{align}
\partial_t \bm{\hat u}_i(t,\bm{x}) + \nabla&\cdot\langle\bm{f}(\mathcal{U}(\bm{\hat u})) \varphi_i\rangle = \bm{0}, \\
\bm{\hat u}_i(t=0,\bm{x}&) = \langle\bm{u}_{\text{IC}}(\bm{x},\cdot)\varphi_i\rangle,
\end{align}
\end{subequations}
with $|i|\leq M$. As already mentioned, the SG moment system is not necessarily hyperbolic. To ensure hyperbolicity, the IPM method uses a solution ansatz which minimizes a given entropy under a moment constraint instead of a polynomial expansion \eqref{eq:SGClosure}. For a given convex entropy $s:\mathbb{R}^m\to\mathbb{R}$ for the original problem \eqref{eq:hyperbolicProblem}, the IPM solution ansatz is given by
\begin{align}\label{eq:primalProblem}
\mathcal{U}(\bm{\hat u}) = \argmin_{\bm u} \langle s(\bm u) \rangle \enskip \text{ subject to } \bm{\hat u}_i = \langle \bm u \varphi_i \rangle \text{ for } |i| \leq M.
\end{align}
Rewritten in its dual form, \eqref{eq:primalProblem} is transformed into an unconstrained optimization problem. Defining the variables $\bm{\lambda}_i\in\mathbb{R}^m$, where $i$ is again a multi index, gives the unconstrained dual problem
\begin{align}\label{eq:dualProblem}
 \bm{\hat \lambda}(\bm{\hat u}) := \argmin_{\bm{\lambda} \in \mathbb{R}^{N \times m}}
  \left\{\langle s_*(\bm{\lambda}^T \bm\varphi)\rangle - \sum_{|i|\leq M}\bm{\lambda}_i^T \bm{\hat u}_i\right\},
\end{align}
where $s_*:\mathbb{R}^m\to\mathbb{R}$ is the Legendre transformation of $s$, and $\bm{ \hat\lambda}:=(\bm{\hat{\lambda}}_i)_{|i|\leq M}\in \mathbb{R}^{N \times m}$ are called the dual variables. The solution to \eqref{eq:primalProblem} is then given by
\begin{align}\label{eq:ansatz}
 \mathcal{U}(\bm{\hat u}) = \left( \nabla_{\bm{u}} s \right)^{-1}(\bm{\hat{\lambda}}(\bm{\hat u})^T \bm{\varphi}).
\end{align}
When plugging this ansatz into the original equations \eqref{eq:hyperbolicProblem} and projecting the resulting residual to zero again yields the moment system \eqref{eq:SGMomentSystem}, but with the ansatz \eqref{eq:ansatz} instead of \eqref{eq:SGClosure}.
\section{Discretization of the IPM system}
\label{sec:framework}
\subsection{Finite Volume Discretization}
In the following, we discretize the moment system in space and time according to \cite{kusch2017maximum}. Due to the fact, that stochastic-Galerkin can be interpreted as IPM with a quadratic entropy, it suffices to only derive a discretization of the IPM moment system. Hence, we discretize the system \eqref{eq:SGMomentSystem} with the more general IPM solution ansatz \eqref{eq:ansatz}.  
Omitting initial conditions and assuming a one-dimensional spatial domain, we can write the IPM system  as
\begin{align*}
\partial_t \bm{\hat u}+\partial_x \bm{F}(\bm{\hat u}) = \bm{0}
\end{align*}
with the flux $\bm{F}:\mathbb{R}^{N\times m}\to\mathbb{R}^{N\times m}$, $\bm{F}(\bm{\hat u})=\langle \bm f(\mathcal{U}(\bm{\hat u}))\bm{\varphi}^T \rangle^T$. Note that the inner transpose represents a dyadic product and therefore the outer transpose is applied to a matrix. Due to hyperbolicity of the IPM moment system, one can use a finite-volume method to approximate the time evolution of the IPM moments. We choose the discrete unknowns which represent the solution to be the spatial averages over each cell at time $t_n$, given by
\begin{align*}
\bm{\hat u}_{ij}^n \simeq \frac{1}{\Delta x}\int_{x_{j-1/ 2}}^{x_{j+1/ 2}}\bm{\hat u}_i(t_n,x) dx.
\end{align*}
If a moment vector in cell $j$ at time $t_n$ is denoted as $\bm{\hat u}_j^n = (\bm{\hat u}_{ij}^n)_{\vert i\vert\leq M}\in\mathbb{R}^{N\times m}$, the finite-volume scheme can be written in conservative form with the numerical flux $\bm{G}:\mathbb{R}^{N\times m}\times\mathbb{R}^{N\times m}\to\mathbb{R}^{N\times m}$ as
\begin{align}\label{eq:IPMDiscretization}
\bm{\hat u}_{j}^{n+1} = \bm{\hat u}_{j}^{n}  - \frac{\Delta t}{\Delta x}\left( \bm{G}(\bm{\hat u}_{j}^{n},\bm{\hat u}_{j+1}^{n})- \bm{G}(\bm{\hat u}_{j-1}^{n},\bm{\hat u}_{j}^{n})\right)
\end{align}
for $j = 1,\cdots,N_x$ and $n = 0,\cdots,N_t$. Here, the number of spatial cells is denoted by $N_x$ and the number of time steps by $N_t$.
The numerical flux is assumed to be consistent, i.e. $\bm{G}(\bm{\hat{u}},\bm{\hat{u}})=\bm{F}(\bm{\hat{u}})$.

When a consistent numerical flux $\bm g:\mathbb{R}^m\times\mathbb{R}^m\to\mathbb{R}^m$, $\bm g = \bm g(\bm u_\ell, \bm u_r)$ is available for the original problem \eqref{eq:hyperbolicProblem}, then for the IPM system we can simply take the numerical flux
\begin{align*}
 \bm{\tilde G}(\bm{\hat u}_{j}^n,\bm{\hat u}_{j+1}^{n}) = \langle \bm g(\mathcal{U}(\bm{\hat u}_j^n),\mathcal{U}(\bm{\hat u}_{j+1}^n))\bm{\varphi}^T\rangle^T
\end{align*}
in \eqref{eq:IPMDiscretization}. Commonly, this integral cannot be evaluated analytically and therefore needs to be approximated by a quadrature rule
\begin{align*}
\langle h \rangle \approx \langle h \rangle_{Q} := \sum_{k=1}^Q w_k h(\bm{\xi}_k)f_{\Xi}(\bm{\xi}_k).
\end{align*}
The approximated numerical flux then becomes
\begin{align}\label{eq:numericalFluxIPM}
 \bm{G}(\bm{\hat u}_{j}^n,\bm{\hat u}_{j+1}^{n}) = \langle \bm g(\mathcal{U}(\bm{\hat u}_j^n),\mathcal{U}(\bm{\hat u}_{j+1}^n))\bm{\varphi}^T\rangle^T_Q.
\end{align}
Note that the numerical flux requires evaluating the ansatz $\mathcal{U}(\bm{\hat u}_j^n)$. To simplify notation, we define $\bm{u}_{s}:\mathbb{R}^m \to \mathbb{R}^m$,
\begin{align*}
\bm{u}_{s}(\bm\Lambda):=\left( \nabla_{\bm{u}} s \right)^{-1}(\bm\Lambda),
\end{align*}
meaning that the IPM ansatz \eqref{eq:ansatz} at cell $j$ in timestep $n$ can be written as
\begin{align*}
\mathcal{U}(\bm{\hat u}_j^n) = \bm{u}_{s}(\bm{\hat{\lambda}}(\bm{\hat u}_j^n)^T \bm{\varphi}).
\end{align*}
The computation of the dual variables $\bm{\hat\lambda}_j^n:=\bm{\hat\lambda}(\bm{\hat u}_j^n)$ requires solving the dual problem \eqref{eq:dualProblem} for the moment vector $\bm{\hat u}_{j}^{n}$. Therefore, to determine the dual variables for a given moment vector $\bm{\hat{u}}$, the cost function
\begin{align}\label{eq:L}
L(\bm{\lambda};\bm{\hat{u}}) := \langle s_*(\bm{\lambda}^T \bm\varphi)\rangle_Q - \sum_{i\leq M}\bm{\lambda}_i^T \bm{\hat u}_i
\end{align}
needs to be minimized. Hence, one needs to find the root of
\begin{align*}
\nabla_{\bm{\lambda}}L(\bm{\lambda};\bm{\hat{u}}) = \langle \nabla s_*(\bm{\lambda}^T \bm\varphi)\bm\varphi^T\rangle_Q^T - \bm{\hat u} = \langle \bm u_s(\bm{\lambda}^T \bm\varphi)\bm\varphi^T\rangle_Q^T - \bm{\hat u},
\end{align*}
where we used $\nabla s_* \equiv \bm u_s$. The root is usually determined by using Newton's method. For simplicity, let us define the full gradient of the Lagrangian to be $\nabla_{\bm{\lambda}}L(\bm{\lambda};\bm{\hat{u}})\in\mathbb{R}^{N\cdot m}$, i.e. we store all entries in a vector. Newton's method uses the iteration function $\bm{d}:\mathbb{R}^{N\times m}\times\mathbb{R}^{N\times m}\to\mathbb{R}^{N\times m}$,
\begin{align}\label{eq:dualIterationFunction}
\bm{d}(\bm{\lambda},\bm{\hat{u}}):= \bm{\lambda}-\bm{H}(\bm{\lambda})^{-1}\cdot\nabla_{\bm{\lambda}}L(\bm{\lambda};\bm{\hat{u}}),
\end{align}
where $\bm H\in\mathbb{R}^{N \cdot m\times N\cdot m}$ is the Hessian of \eqref{eq:L}, given by
\begin{align*}
\bm{H}(\bm{\lambda}) := \langle \nabla \bm{u}_{s} (\bm{\lambda}^T\bm{\varphi})\otimes\bm{\varphi}\bm{\varphi}^T\rangle_Q^{T}.
\end{align*}
The function $\bm d$ will in the following be called dual iteration function. Now, the Newton iteration $l$ for spatial cell $j$ is given by
\begin{align}\label{eq:dualIteration1}
\bm{\lambda}^{(l+1)}_j = \bm{d}(\bm{\lambda}_j^{(l)},\bm{\hat{u}}_j).
\end{align}
The exact dual state is then obtained by computing the fixed point of $\bm{d}$, meaning that one converges the iteration \eqref{eq:dualIteration1}, i.e. $\bm{\hat\lambda}_j^n:=\bm{\hat\lambda}(\bm{\hat u}_j^n)=\lim_{l\rightarrow\infty}\bm{d}(\bm{\lambda}_j^{(l)},\bm{\hat{u}}_j^n)$.
To obtain a finite number of iterations for the iteration in cell $j$, the stopping criterion 
\begin{align}\label{eq:tauCrit}
\sum_{i=0}^m\left\Vert \nabla_{\bm{\lambda_i}}L(\bm{\lambda}_j^{(l)};\bm{\hat{u}}_j^n) \right\Vert < \tau
\end{align}
is used.

We now write down the entire scheme: To obtain a more compact notation, we define
\begin{align}\label{eq:momentIterationFunction}
\bm{c}\left(\bm{\lambda}_{\ell},\bm{\lambda}_c,\bm{\lambda}_r\right):= \langle \bm u_{s}(\bm{\lambda}_c^T\bm{\varphi})\bm{\varphi}^T\rangle_Q^T - \frac{\Delta t}{\Delta x}\left(\langle \bm g(\bm u_{s}(\bm{\lambda}_c^T\bm{\varphi}),\bm u_{s}(\bm{\lambda}_r^T\bm{\varphi}))\bm{\varphi}^T\rangle_Q^T-\langle \bm g(\bm u_{s}(\bm{\lambda}_{\ell}^T\bm{\varphi}),\bm u_{s}(\bm{\lambda}_c^T\bm{\varphi}))\bm{\varphi}^T\rangle_Q^T\right).
\end{align}
The moment iteration is then given by
\begin{align}\label{eq:momentIteration}
\bm{\hat u}_j^{n+1} = \bm{c}\left(\bm{\hat\lambda}(\bm{\hat u}_{j-1}^n),\bm{\hat\lambda}(\bm{\hat u}_{j}^n),\bm{\hat\lambda}(\bm{\hat u}_{j+1}^n)\right),
\end{align}
where the map from the moment vector to the dual variables, i.e. $\bm{\hat\lambda}(\bm{\hat u}_{j}^n)$, is obtained by iterating
\begin{align}\label{eq:dualIteration}
\bm{\lambda}_j^{(l+1)} = \bm{d}(\bm{\lambda}_{j}^{(l)};\bm{\hat u}_j^{n}).
\end{align}
until condition \eqref{eq:tauCrit} is fulfilled. This gives Algorithm \ref{alg:IPM}.

\begin{algorithm}[H]
\begin{algorithmic}[1]
\For{$j=0$ to $N_x+1$}
\State $\bm{u}_j^0 \leftarrow \frac{1}{\Delta x} \int_{x_{j-1/ 2}}^{x_{j+1/ 2}} \langle u_{\text{IC}}(x, \cdot) \bm{\varphi} \rangle_Q dx$
\EndFor
\For{$n=0$ to $N_t$}
\For{$j=0$ to $N_x+1$}
\State $\bm{\lambda}_j^{(0)} \leftarrow \bm{\hat \lambda}_j^{n}$
\While{\eqref{eq:tauCrit} is violated}
\State $\bm{\lambda}_j^{(l+1)} \leftarrow \bm{d}(\bm{\lambda}_{j}^{(l)};\bm{\hat u}_j^{n})$
\State $l \leftarrow l+1$
\EndWhile
\State $\bm{\hat \lambda}_j^{n+1} \leftarrow \bm{\lambda}_j^{(l)}$
\EndFor
\For{$j=1$ to $N_x$}
\State $\bm{\hat u}_j^{n+1} \leftarrow \bm{c}(\bm{\hat \lambda}_{j-1}^{n+1},\bm{\hat \lambda}_j^{n+1},\bm{\hat \lambda}_{j+1}^{n+1})$
\EndFor
\EndFor
\end{algorithmic}
\caption{IPM algorithm}
\label{alg:IPM}
\end{algorithm}

\subsection{Properties of the kinetic flux}
\label{sec:costNumFlux}

A straight-forward implementation is ensured by the choice of the numerical flux \eqref{eq:numericalFluxIPM}. This choice of the numerical flux is common in the field of transport theory, where it is called the \textit{kinetic flux} or \textit{kinetic scheme}, see e.g. \cite{deshpande1986kinetic,harten1983upstream,perthame1990boltzmann,perthame1992second}. By simply taking moments of a given numerical flux for the deterministic problem, the method can easily be applied to various physical problems whenever an implementation of $\bm g = \bm g(\bm u_\ell, \bm u_r)$ is available. Therefore, we call the proposed numerical method \textit{semi-intrusive}.

Intrusive numerical methods which compute arising integrals analytically and therefore directly depend on the moments (i.e. they do not necessitate the evaluation of the gPC expansion on quadrature points) can be constructed by performing a gPC expansion on the system flux directly \cite{debusschere2004numerical}. Examples can be found in \cite{hu2015stochastic,hu2016stochastic,tryoen2010instrusive,durrwachterahigh} for the computation of numerical fluxes and sources. While the analytic computation of arising integrals is not always more efficient \cite[Section 6]{ghanem1998stochastic}, it can also complicate recycling a deterministic solver. See \ref{app:costNumFlux} for a comparison of numerical costs when using Burgers' equation. However, when not using a quadratic entropy in the IPM method or when the physical flux of the deterministic problem is not a polynomial, it is not clear how many quadrature points the numerical quadrature rule requires to guarantee a sufficiently small quadrature error. We will study the approximation properties of IPM with different quadrature orders in Section~\ref{sec:resultsNACA1D}.
\section{One-Shot IPM}
\label{sec:OneShotIPM}

In the following section we only consider steady state problems, i.e. equation \eqref{eq:fulleq} reduces to
\begin{align}\label{eq:hyperbolicProblemSteady}
\nabla\cdot\bm{f}(\bm{u}(\bm{x},\bm{\xi})) = \bm{0} \enskip \text{ in } D
\end{align}
with adequate boundary conditions. A general strategy for computing the steady state solution to \eqref{eq:hyperbolicProblemSteady} is to introduce a pseudo-time and numerically treat \eqref{eq:hyperbolicProblemSteady} as an unsteady problem. A steady state solution is then obtained by iterating in pseudo-time until the solution remains constant. It is important to point out that the time it takes to converge to a steady state solution is crucially affected by the chosen initial condition and its distance to the steady state solution.
Similar to the unsteady case \eqref{eq:hyperbolicProblem}, we can again derive a moment system for \eqref{eq:hyperbolicProblemSteady} given by
\begin{align}\label{eq:MomentSystemSteady}
\nabla\cdot\langle\bm{f}(\bm{u}(\bm{x},\bm{\xi}))\bm{\varphi}^T\rangle^T = \bm{0} \enskip \text{ in } D
\end{align}
which is again needed for the construction of intrusive methods. By introducing a pseudo-time and using the IPM closure, we obtain the same system as in \eqref{eq:SGMomentSystem}, i.e. Algorithm \ref{alg:IPM} can be used to iterate to a steady state solution. Note that now, the time iteration is not performed for a fixed number of time steps $N_t$, but until the condition
\begin{align}\label{eq:residualSteady}
\sum_{j = 1}^{N_x} \Delta x_j \Vert \bm{\hat{u}}_j^n - \bm{\hat{u}}_j^{n-1} \Vert \leq \varepsilon
\end{align}
is fulfilled. Condition \eqref{eq:residualSteady}, which is for example being used in the SU2 code framework \cite{economon2015su2}, measures the change of the solution by a single time iteration. Note, that in order to obtain an estimate of the distance to the steady state solution, one has to include the Lipschitz constant of the corresponding fixed point problem. Since one is generally interested in low order moments such as the expectation value, the residual \eqref{eq:residualSteady} can be modified by only accounting for the zero order moments.

In this section we aim at breaking up the inner loop in the IPM Algorithm \ref{alg:IPM}, i.e. to just perform one iteration of the dual problem in each time step. Consequently, the IPM reconstruction given by \eqref{eq:primalProblem} is not done exactly, meaning that the reconstructed solution does not minimize the entropy while not fulfilling the moment constraint. However, the fact that the moment vectors are not yet converged to the steady solution seems to permit such an inexact reconstruction. Hence, we aim at iterating the moments to steady state and the dual variables to the exact solution of the IPM optimization problem \eqref{eq:primalProblem} simultaneously.
By successively performing one update of the moment iteration and one update of the dual iteration, we obtain 
\begin{subequations}\label{eq:oneshotIPM}
\begin{align}
&\bm{\lambda}_{j}^{n+1} =  \bm{d}(\bm{\lambda}_j^{n},\bm{u}_j^{n}) \enskip \text{ for all }j \label{eq:oneshotIPMdual}\\
&\bm{u}_j^{n+1} =  \bm{c}\left(\bm{\lambda}_{j-1}^{n+1},\bm{\lambda}_{j}^{n+1},\bm{\lambda}_{j+1}^{n+1}\right) \enskip \text{ for all }j \label{eq:oneshotIPMmoment}.
\end{align}
\end{subequations}
This yields Algorithm \ref{alg:osIPM}.
\begin{algorithm}[H]
\begin{algorithmic}[1]
\For{$j=0$ to $N_x+1$}
\State $\bm{u}_j^0 \leftarrow \frac{1}{\Delta x} \int_{x_{j-1/ 2}}^{x_{j+1/ 2}} \langle u_{\text{IC}}(x, \cdot) \bm{\varphi} \rangle_Q dx$
\EndFor
\While{\eqref{eq:residualSteady} is violated}
\For{$j=1$ to $N_x$}
\State $\bm{\lambda}_j^{n+1} \leftarrow \bm{d}(\bm{\lambda}_{j}^{n};\bm{\hat u}_j^{n})$
\EndFor
\For{$j=1$ to $N_x$}
\State $\bm{\hat u}_j^{n+1} \leftarrow \bm{c}(\bm{\lambda}_{j-1}^{n+1},\bm{\lambda}_j^{n+1},\bm{\lambda}_{j+1}^{n+1})$
\EndFor
\State $n \leftarrow n+1$
\EndWhile
\end{algorithmic}
\caption{One-Shot IPM implementation}
\label{alg:osIPM}
\end{algorithm}
We call this method One-Shot IPM, since it is inspired by One-Shot optimization, see for example \cite{hazra2005aerodynamic}, which uses only a single iteration of the primal and dual step in order to update the design variables. Note that the dual variables from the One-Shot IPM iteration are written without a hat to indicate that they are not the exact solution of the dual problem.

In the following, we will show that this iteration converges, if the chosen initial condition is sufficiently close to the steady state solution. For this we take an approach commonly chosen to prove local convergence properties of Newton's method: In Theorem \ref{th:Contractive}, we show that the iteration function is contractive at its fixed point and conclude in Theorem \ref{th:localConvergence} that this yields local convergence. Hence, we preserve the convergence property of the original IPM method, which uses Newton's method and therefore only converges locally as well.
\begin{theorem}\label{th:Contractive}
Assume that the classical IPM iteration is contractive at its fixed point $\bm{\hat u}^*$. Then the Jacobian $\bm{J}$ of the One-Shot IPM iteration \eqref{eq:oneshotIPM} has a spectral radius $\rho(\bm{J})<1$ at the fixed point $(\bm{\lambda}^*,\bm{\hat u}^*)$.
\end{theorem}
\begin{proof}
First, to understand what contraction of the classical IPM iteration implies, we rewrite the moment iteration \eqref{eq:momentIteration} of the classical IPM scheme: When defining the update function
\begin{align*}
\bm{\tilde c}\left(\bm{\hat{u}}_{\ell},\bm{\hat{u}}_{c},\bm{\hat{u}}_{r}\right):=\bm{c}\left(\bm{\hat{\lambda}}(\bm{\hat{u}}_{\ell}),\bm{\hat{\lambda}}(\bm{\hat{u}}_{c}),\bm{\hat{\lambda}}(\bm{\hat{u}}_{r})\right)
\end{align*}
we can rewrite the classical moment iteration as
\begin{align}\label{eq:shortIPMIt}
\bm{\hat u}_j^{n+1} = \bm{\tilde c}\left(\bm{\hat u}_{j-1}^n,\bm{\hat u}_{j}^n,\bm{\hat u}_{j+1}^n\right).
\end{align}
Since we assume that the classical IPM scheme is contractive at its fixed point, we have $\rho (\nabla_{\bm{\hat u}}\bm{\tilde c}(\bm{\hat u}^*))<1$ with $\nabla_{\bm{\hat u}}\bm{\tilde c}\in\mathbb{R}^{N\cdot N_x\times N\cdot N_x}$ defined by
\begin{align*}
\nabla_{\bm{\hat u}}\bm{\tilde c} = 
\begin{pmatrix} 
    \partial_{\bm{\hat u}_c}\bm{\tilde c}_{1} & \partial_{\bm{\hat u}_r}\bm{\tilde c}_{1}& 0 & 0 & \dots \\
    \partial_{\bm{\hat u}_{\ell}}\bm{\tilde c}_{2} & \partial_{\bm{\hat u}_c}\bm{\tilde c}_{2} & \partial_{\bm{\hat u}_r}\bm{\tilde c}_{2}& 0 & \dots \\
    0 & \partial_{\bm{\hat u}_{\ell}}\bm{\tilde c}_{3} & \partial_{\bm{\hat u}_c}\bm{\tilde c}_{3} & \partial_{\bm{\hat u}_r}\bm{\tilde c}_{3}\\
    \vdots & & & \ddots & \\
    0 &\cdots &  0 & \partial_{\bm{\hat u}_{\ell}}\bm{\tilde c}_{N_x} & \partial_{\bm{\hat u}_c}\bm{\tilde c}_{N_x}
    \end{pmatrix},
\end{align*}
where we define $\bm{\tilde c}_{j}:=\bm{\tilde c}\left(\bm{\hat u}_{j-1}^*,\bm{\hat u}_{j}^*,\bm{\hat u}_{j+1}^*\right)$ for all $j$. Now for each term inside the matrix $\nabla_{\bm{\hat u}}\bm{\tilde c}$ we have 
\begin{align}\label{eq:cTildeDer}
\partial_{\bm{\hat u}_{\ell}}\bm{\tilde c}_{j} = \frac{\partial \bm{c}_j}{\partial \bm{\hat \lambda}_{\ell}}\frac{\partial \bm{\hat \lambda}(\bm{\hat u}_{j-1}^*)}{\partial \bm{\hat u}},\enskip\partial_{\bm{\hat u}_c}\bm{\tilde c}_{j} = \frac{\partial \bm{c}_j}{\partial \bm{\hat \lambda}_c}\frac{\partial \bm{\hat \lambda}(\bm{\hat u}_j^*)}{\partial \bm{\hat u}},\enskip\partial_{\bm{\hat u}_r}\bm{\tilde c}_{j} = \frac{\partial \bm{c}_j}{\partial \bm{\hat \lambda}_r}\frac{\partial \bm{\hat \lambda}(\bm{\hat u}_{j+1}^*)}{\partial \bm{\hat u}}.
\end{align}
We first wish to understand the structure of the terms $\partial_{\bm{\hat u}} \bm{\hat \lambda}(\bm{\hat u})$. For this, we note that the exact dual variables fulfill
\begin{align}\label{eq:ulambda}
\bm{\hat u} = \langle \bm{u}_s(\bm{\hat \lambda}^T\bm{\varphi})\bm{\varphi}^T\rangle^T =: \bm{h}(\bm{\hat \lambda}),
\end{align}
which is why we have the mapping $\bm{\hat u}:\mathbb{R}^{N\times m}\to\mathbb{R}^{N\times m}$, $\bm{\hat u}(\bm{\hat \lambda}) = \bm{h}(\bm{\hat \lambda})$. Since the solution of the dual problem for a given moment vector is unique, this mapping is bijective and therefore we have an inverse function
\begin{align}\label{eq:lambdau}
\bm{\hat \lambda} = \bm{h}^{-1}(\bm{\hat u}(\bm{\hat \lambda})).
\end{align}
Now we differentiate both sides w.r.t. $\bm{\hat \lambda}$ to get
\begin{align*}
\bm{I}_{d} = \frac{\partial \bm{h}^{-1}(\bm{\hat u}(\bm{\hat \lambda}))}{\partial \bm{\hat u}}\frac{\partial \bm{\hat u}(\bm{\hat \lambda})}{\partial \bm{\hat \lambda}}.
\end{align*}
We multiply with the matrix inverse of $\frac{\partial \bm{\hat u}(\bm{\hat \lambda})}{\partial \bm{\hat \lambda}}$ to get
\begin{align*}
\left(\frac{\partial \bm{\hat u}(\bm{\hat \lambda})}{\partial \bm{\hat \lambda}}\right)^{-1} = \frac{\partial \bm{h}^{-1}(\bm{\hat u}(\bm{\hat \lambda}))}{\partial \bm{\hat u}}.
\end{align*}
Note that on the left-hand-side we have the inverse of a matrix and on the right-hand-side, we have the inverse of a multi-dimensional function. By rewriting $\bm{h}^{-1}(\bm{\hat u}(\bm{\hat \lambda}))$ as $\bm{\hat \lambda}(\bm{\hat u})$ and simply computing the term $\frac{\partial \bm{\hat u}(\bm{\hat \lambda})}{\partial \bm{\hat \lambda}}$ by differentiating \eqref{eq:ulambda} w.r.t. $\bm{\hat \lambda}$, one obtains
\begin{align}\label{eq:dudlambdaex}
\partial_{\bm{\hat u}} \bm{\hat \lambda}(\bm{\hat u}) = \langle \nabla\bm{u}_s(\bm{\hat \lambda}^T\bm{\varphi})\bm{\varphi}\bm{\varphi}^T\rangle^{-T}.
\end{align}
Now we begin to derive the spectrum of the \textit{One-Shot IPM} iteration \eqref{eq:oneshotIPM}. Note that in its current form this iteration is not really a fixed point iteration, since it uses the time updated dual variables in \eqref{eq:oneshotIPMmoment}. To obtain a fixed point iteration, we plug the dual iteration step \eqref{eq:oneshotIPMdual} into the moment iteration \eqref{eq:oneshotIPMmoment} to obtain
\begin{align*}
&\bm{\lambda}_j^{n+1} = \bm{d}(\bm{\lambda}_j^{n},\bm{\hat u}_j^{n}) \enskip \text{ for all j} \\
&\bm{\hat u}_j^{n+1} =  \bm{c}\left(\bm{d}(\bm{\lambda}_{j-1}^{n},\bm{\hat u}_{j-1}^{n}),\bm{d}(\bm{\lambda}_{j}^{n},\bm{\hat u}_{j}^{n}),\bm{d}(\bm{\lambda}_{j+1}^{n},\bm{\hat u}_{j+1}^{n})\right).
\end{align*}
The Jacobian $\bm{J}\in\mathbb{R}^{2N\cdot N_x \times 2N\cdot N_x}$ has the form
\begin{align}\label{eq:Jacobian}
\bm{J} = 
\begin{pmatrix}
 \partial_{\bm{\lambda}} \bm{d} & \partial_{\bm{\hat u}} \bm{d}  \\
\partial_{\bm{\lambda}} \bm{c} & \partial_{\bm{\hat u}} \bm{c}
\end{pmatrix},
\end{align}
where each block has entries for all spatial cells. We start by looking at $\partial_{\bm{\lambda}} \bm{d}$. For the columns belonging to cell $j$, we have
\begin{align*}
\partial_{\bm{\lambda}} \bm{d}(\bm{\lambda}_j^n,\bm{\hat u}_j^n) &= \bm{I}_d - \bm{H}(\bm\lambda)^{-1} \cdot \langle \nabla\bm{u}_s(\bm{\varphi}^T\bm{\lambda}_j^n)\bm{\varphi}\bm{\varphi}^T \rangle^T - \partial_{\bm{\lambda}}\bm{H}(\bm\lambda)^{-1} \cdot \left( \langle \bm{u}_s(\bm{\varphi}^T\bm{\lambda}_j^n)\bm{\varphi}^T \rangle^T - \bm{\hat u}\right) \\
&=- \partial_{\bm{\lambda}}\bm{H}(\bm\lambda)^{-1} \cdot \left( \langle \bm{u}_s(\bm{\varphi}^T\bm{\lambda}_j^n)\bm{\varphi}^T \rangle^T - \bm{\hat u}\right).
\end{align*}
Recall that at the fixed point $(\bm{\lambda}^*,\bm{\hat u}^*)$, we have $\langle \bm{u}_s(\bm{\varphi}^T\bm{\lambda}_j^n)\bm{\varphi}^T \rangle^T = \bm{\hat u}$, hence one obtains $\partial_{\bm{\lambda}} \bm{d}=\bm{0}$. For the block $\partial_{\bm{\hat u}} \bm{d}$, we get 
\begin{align*}
\partial_{\bm{\hat u}} \bm{d}(\bm{\lambda}_j^n,\bm{\hat u}_j^n) = \bm{H}(\bm\lambda)^{-1},
\end{align*}
hence $\partial_{\bm{\hat u}} \bm{d}$ is a block diagonal matrix. Let us now look at $\partial_{\bm{\lambda}} \bm{c}$ at a fixed spatial cell $j$:
\begin{align*}
\frac{\partial \bm{c}}{\partial \bm{\lambda}_{\ell}}\frac{\partial \bm{d}(\bm{\lambda}_{j-1}^{n},\bm{\hat u}_{j-1}^{n})}{\partial \bm{\lambda}} = \bm{0},
\end{align*}
since we already showed that by the choice of $\bm{H}(\bm\lambda)^{-1}$ the term $\partial_{\bm{\lambda}} \bm{d}$ is zero. We can show the same result for all spatial cells and all inputs of $\bm{c}$ analogously, hence $\partial_{\bm{\lambda}} \bm{c} = \bm{0}$. For the last block, we have that 
\begin{align*}
\frac{\partial \bm{c}}{\partial \bm{\lambda}_{\ell}}\frac{\partial \bm{d}(\bm{\lambda}_{j-1}^{n},\bm{\hat u}_{j-1}^{n})}{\partial \bm{\hat u}} = \frac{\partial \bm{c}}{\partial \bm{\lambda}_{\ell}} \bm{H}(\bm\lambda)^{-1} = \frac{\partial \bm{c}}{\partial \bm{\lambda}_{\ell}} \langle \nabla\bm{u}_s(\bm{\varphi}^T\bm{\lambda}_{j-1}^n)\bm{\varphi}\bm{\varphi}^T \rangle^{-T} = \partial_{\bm{\hat u}_{\ell}}\bm{\tilde c}_j
\end{align*}
by the choice of $\bm{H}(\bm\lambda)^{-1}$ as well as \eqref{eq:cTildeDer} and \eqref{eq:dudlambdaex}. We obtain an analogous result for the second and third input. Hence, we have that $\partial_{\bm{\hat u}} \bm{c} = \nabla_{\bm{\hat u}}\bm{\tilde c}$, which only has eigenvalues between $-1$ and $1$ by the assumption that the classical IPM iteration is contractive. Since $\bm{J}$ is an upper triangular block matrix, the eigenvalues are given by $\lambda\left(\partial_{\bm{\lambda}} \bm{d}\right) = 0$ and $\lambda\left(\partial_{\bm{\hat u}} \bm{c}\right)\in(-1,1)$, hence the One-Shot IPM is contractive around its fixed point.
\end{proof}
\begin{theorem}\label{th:localConvergence}
With the assumptions from Theorem \ref{th:Contractive}, the One-Shot IPM converges locally, i.e. there exists a $\delta>0$ s.t. for all starting points $(\bm{\lambda}^0,\bm{\hat u}^0)\in B_{\delta}(\bm{\lambda}^*,\bm{\hat u}^*)$ we have
\begin{align*}
\Vert (\bm{\lambda}^n,\bm{\hat u}^n) - (\bm{\lambda}^*,\bm{\hat u}^*)\Vert \rightarrow 0 \qquad \text{ for } n \rightarrow \infty.
\end{align*}
\end{theorem}
\begin{proof}
By Theorem \ref{th:Contractive}, the One-Shot scheme is contractive at its fixed point. Since we assumed convergence of the classical IPM scheme, we can conclude that all entries in the Jacobian $\bm{J}$ are continuous functions. Furthermore, the determinant of $\bm{\tilde{J}}:=\bm{J}-\lambda \bm{I}_d$ is a polynomial of continuous functions, since
\begin{align*}
\text{det}(\bm{\tilde J}) = \sum_{\sigma} \text{sgn}(\sigma)\prod_{i = 1}^{2 N_x N} \tilde J_{\sigma(i),i}.
\end{align*}
Since the roots of a polynomial vary continuously with its coefficients, the eigenvalues of $\bm{J}$ are continuous w.r.t $(\bm{\lambda},\bm{\hat u})$. Hence there exists an open ball with radius $\delta$ around the fixed point in which the eigenvalues remain in the interval $(-1,1)$.
\end{proof}
\section{Adaptivity}
\label{sec:adaptivity}

The following section presents the adaptivity strategy used in this work. Since stochastic hyperbolic problems generally experience shocks in a small portion of the space-time domain, the idea is to perform arising computations on a high accuracy level in this small area, while keeping a low level of accuracy in the remainder. The idea is to automatically select the lowest order moment capable of approximating the solution with given accuracy, i.e. the same error is obtained while using a significantly reduced number of unknowns in most parts of the computational domain and thus boost the performance of intrusive methods.

In the following, we discuss the building blocks of the IPM method for refinement levels $\ell = 1,\cdots,N_{\text{ad}}$, where level $1$ uses the coarsest discretization and level $N_{\text{ad}}$ uses the finest discretization of the uncertain domain. At a given refinement level $\ell$, the total degree of the basis function is given by $M_{\ell}$ with a corresponding number of moments $N_{\ell}$. The number of quadrature points at level $\ell$ is denoted by $Q_{\ell}$. To determine the refinement level of a given moment vector $\bm{\hat u}$ we choose techniques used in discontinuous Galerkin (DG) methods. Adaptivity is a common strategy to accelerate this class of methods and several indicators to determine the smoothness of the solution exist. Translating the idea of the so-called discontinuity sensor which has been defined in \cite{persson2006sub} to uncertainty quantification, we define the polynomial approximation at refinement level $\ell$ as
\begin{align*}
\bm{\tilde u}_{\ell} := \sum_{|i|\leq M_{\ell}} \bm{\hat{u}}_i \varphi_i.
\end{align*}
Now the indicator for a moment vector at level $\ell$ is defined as
\begin{align}\label{eq:errorIndicator}
\bm S_{\ell} := \frac{\langle \left(\bm{\tilde u}_{\ell} - \bm{\tilde u}_{\ell-1}\right)^2\rangle}{\langle \bm{\tilde u}_{\ell}^2\rangle},
\end{align}
where divisions and multiplications are performed element-wise. Note that a similar indicator has been used in \cite{kroker2012finite} for intrusive methods in uncertainty quantification. In this work, we use the first entry in $\bm S_{\ell}$ to determine the refinement level, i.e. in the case of gas dynamics, the regularity of the density is chosen to indicate an adequate refinement level. If the moment vector in a given cell and at a certain timestep is initially at refinement level $\ell$, this level is kept if the error indicator \eqref{eq:errorIndicator} lies in the interval $I_{\delta}:=[\delta_{-},\delta_{+}]$. Here $\delta_{\pm}$ are user determined parameters. If the indicator is smaller than $\delta_-$, the refinement level is decreased to the next lower level, if it lies above $\delta_+$, it is increased to the next higher level.

Now we need to specify how the different building blocks of IPM can be modified to work with varying truncation orders in different cells. Let us first add dimensions to the notation of the dual iteration function $\bm d$, which has been defined in \eqref{eq:dualIterationFunction}. Now, we have 
$\bm{d}_{\ell}:\mathbb{R}^{N_{\ell}\times m}\times\mathbb{R}^{N_{\ell}\times m}\to\mathbb{R}^{N_{\ell}\times m}$, given by
\begin{align}\label{eq:dualIterationFunctionAd}
\bm{d}_{\ell}(\bm{\lambda},\bm{\hat{u}}):= \bm{\lambda}-\bm{H}_{\ell}^{-1}(\bm{\lambda})\cdot \left(\langle \bm u_{s}(\bm{\lambda}^T\bm{\varphi}_{\ell})\bm{\varphi}_{\ell}^T\rangle_{Q_{\ell}}^T-\bm{\hat{u}}\right),
\end{align}
where $\bm{\varphi}_{\ell}\in\mathbb{R}^{N_{\ell}}$ collects all basis functions with total degree smaller or equal to $M_\ell$. The Hessian $\bm{H}_{\ell}$ is given by 
\begin{align*}
\bm{H}_{\ell}(\bm{\lambda}) := \langle \nabla \bm{u}_{s} (\bm{\lambda}^T\bm{\varphi}_{\ell})\otimes\bm{\varphi}_{\ell}\bm{\varphi}_{\ell}^T\rangle_{Q_{\ell}}^{T}.
\end{align*}
An adaptive version of the moment iteration \eqref{eq:momentIterationFunction} is denoted by $\bm c_{\ell}^{\bm{\ell}'}:\mathbb{R}^{N_{\ell_1'}\times m}\times \mathbb{R}^{N_{\ell_2'}\times m}\times \mathbb{R}^{N_{\ell_3'}\times m}\rightarrow \mathbb{R}^{N_{\ell}\times m}$ and given by
\begin{align}\label{eq:adaptiveFVUpdate}
\bm{c}_{\ell}^{\bm{\ell}'}\left(\bm{\lambda}_{1},\bm{\lambda}_2,\bm{\lambda}_3\right):= &\langle \bm u_{s}(\bm{\lambda}_2^T\bm{\varphi}_{\ell_2'})\bm{\varphi}_{\ell}^T\rangle_{Q_{\ell}}^T \\&- \frac{\Delta t}{\Delta x}\left(\langle \bm g(\bm u_{s}(\bm{\lambda}_2^T\bm{\varphi}_{\ell_2'}),\bm u_{s}(\bm{\lambda}_3^T\bm{\varphi}_{\ell_3'}))\bm{\varphi}_{\ell}^T\rangle_{Q_{\ell}}^T-\langle \bm g(\bm u_{s}(\bm{\lambda}_{1}^T\bm{\varphi}_{\ell_1'}),\bm u_{s}(\bm{\lambda}_2^T\bm{\varphi}_{\ell_2'}))\bm{\varphi}_{\ell}^T\rangle_{Q_{\ell}}^T\right). \nonumber
\end{align}
Hence, the index vector $\bm\ell'\in\mathbb{N}^{3}$ denotes the refinement levels of the stencil cells, which are used to compute the time updated moment vector at level $\ell$.

The strategy now is to perform the dual update for a set of moment vectors $\bm{\hat u}_j^n$ at refinement levels $\ell_j^n$ for $j = 1,\cdots,N_x$. Thus, the dual iteration makes use of the iteration function \eqref{eq:dualIterationFunctionAd} at refinement level $\ell_j^n$. After that, the refinement level at the next time step $\ell_j^{n+1}$ is determined by making use of the smoothness indicator \eqref{eq:errorIndicator}. The moment update then computes the moments at the time updated refinement level $\ell_j^{n+1}$, utilizing the the dual states at the old refinement levels $\bm{\ell}' = (\ell_{j-1}^n,\ell_{j}^n,\ell_{j+1}^n)^T$. 

Note that we use nested quadrature rules, which facilitate the task of evaluating the quadrature in the moment update \eqref{eq:adaptiveFVUpdate}. Assume that we want to compute the moment update in cell $j$ with refinement level $\ell_j$ where a neighboring cell $j-1$ has refinement level $\ell_{j-1}$. Now if $\ell_{j-1}\geq\ell_j$, the solution of cell $j-1$ is known at all $Q_{\ell}$ quadrature points, hence the integral inside the moment update can be computed. Vice versa, if $\ell_{j-1}\leq\ell_j$, we need to evaluate the neighboring cell at the finer quadrature level $\ell_j$. Except from this, increasing or decreasing the refinement level does not lead to additional costs.

The IPM algorithm with adaptivity results in Algorithm \ref{alg:ad-IPM}.
\begin{algorithm}[H]
\begin{algorithmic}[1]
\For{$j=0$ to $N_x+1$}
\State $\ell_j^0 \leftarrow$ choose initial refinement level
\State $\bm{u}_j^0 \leftarrow \frac{1}{\Delta x} \int_{x_{j-1/ 2}}^{x_{j+1/ 2}} \langle u_{\text{IC}}(x, \cdot) \bm{\varphi}_{\ell_j^0} \rangle_{Q_{\ell_j^0}} dx$
\EndFor
\For{$n=0$ to $N_t$}
\For{$j=0$ to $N_x+1$}
\State $\bm{\lambda}_j^{(0)} \leftarrow \bm{\hat \lambda}_j^{n}$
\While{\eqref{eq:tauCrit} is violated}
\State $\bm{\lambda}_j^{(l+1)} \leftarrow \bm{d}_{\ell_j^n}(\bm{\lambda}_{j}^{(l)};\bm{\hat u}_j^{n})$
\State $l \leftarrow l+1$
\EndWhile
\State $\bm{\hat \lambda}_j^{n+1} \leftarrow \bm{\lambda}_j^{(l)}$
\State $\ell_j^{n+1}\leftarrow \text{DetermineRefinementLevel}\left(\bm{\hat \lambda}_j^{n+1}\right)$
\EndFor
\For{$j=1$ to $N_x$}
\State $\bm\ell' \leftarrow (\ell_{j-1}^n,\ell_{j}^n,\ell_{j+1}^n)^T$
\State $\bm{\hat u}_j^{n+1} \leftarrow \bm{c}_{\ell_j^{n+1}}^{\bm\ell'}(\bm{\hat \lambda}_{j-1}^{n+1},\bm{\hat \lambda}_j^{n+1},\bm{\hat \lambda}_{j+1}^{n+1})$
\EndFor
\EndFor
\end{algorithmic}
\caption{Adaptive IPM implementation}
\label{alg:ad-IPM}
\end{algorithm}
Adaptivity can be used for intrusive methods in general as well as for steady and unsteady problems. In the case of steady problems, we can make use of a strategy, which we call \textit{refinement retardation}. Recall that the convergence to an admissible steady state solution is expensive and a high accuracy and desirable solution properties are only required at the end of this iteration process. Hence, we propose to iteratively increase the maximal refinement level whenever the residual \eqref{eq:residualSteady} lies below a certain tolerance $\varepsilon$. For a given set of maximal refinement levels $\ell_l^*$ and a set of tolerances $\varepsilon_l^*$ at which the refinement level must be increased, we can now perform a large amount of the required iterations on a lower, but cheaper refinement level.
The same strategy can be applied for One-Shot IPM. In this case, the algorithm is given by Algorithm~\ref{alg:adosIPM}.
\begin{algorithm}[H]
\begin{algorithmic}[1]
\For{$j=0$ to $N_x+1$}
\State $\bm{u}_j^0 \leftarrow \frac{1}{\Delta x} \int_{x_{j-1/ 2}}^{x_{j+1/ 2}} \langle u_{\text{IC}}(x, \cdot) \bm{\varphi} \rangle_Q dx$
\EndFor
\While{\eqref{eq:residualSteady} is violated}
\For{$j=1$ to $N_x$}
\State $\bm{\lambda}_j^{n+1} \leftarrow \bm{d}_{\ell_j^n}(\bm{\lambda}_{j}^{n};\bm{\hat u}_j^{n})$
\State $\ell_j^{n+1}\leftarrow \max\{\text{DetermineRefinementLevel}\left(\bm{\lambda}_j^{n+1}\right),\ell_l^*\}$
\EndFor
\For{$j=1$ to $N_x$}
\State $\bm\ell' \leftarrow (\ell_{j-1}^n,\ell_{j}^n,\ell_{j+1}^n)^T$
\State $\bm{\hat u}_j^{n+1} \leftarrow \bm{c}_{\ell_j^{n+1}}^{\bm\ell'}(\bm{\lambda}_{j-1}^{n+1},\bm{\lambda}_j^{n+1},\bm{\lambda}_{j+1}^{n+1})$
\EndFor
\State $n \leftarrow n+1$
\If{the residual \eqref{eq:residualSteady} lies below $\varepsilon_l^*$}
\State $l \leftarrow l+1$
\EndIf
\EndWhile
\end{algorithmic}
\caption{Adaptive One-Shot IPM implementation with refinement retardation}
\label{alg:adosIPM}
\end{algorithm}

\section{Parallelization and Implementation}
\label{sec:parallel}
It remains to discuss the parallelization of the presented algorithms. In order to minimize the parallelization overhead, our goal is to minimize the communication between processors. Note that the dual problem (line 8 in Algorithm~\ref{alg:ad-IPM} and line 5 in Algorithm~\ref{alg:adosIPM}) does not require communication, i.e. it suffices to distribute the spatial cells between processors. In contrast to that, the finite volume update (line 14 in Algorithm~\ref{alg:ad-IPM} and line 8 in Algorithm~\ref{alg:adosIPM}) requires communication, since values at neighboring cells need to be evaluated. Hence, distributing the spatial mesh between processors will yield communication overhead since data needs to be sent whenever a stencil cell lies on a different processor. Therefore, we choose to parallelize the quadrature points, which minimizes the computational time spend on communication. As mentioned in \ref{app:costNumFlux}, we first compute the solution at stencil cells for all quadrature points. I.e. we determine $\bm u^{(j)}_k\in\mathbb{R}^m$ and the corresponding stencil cells for $k = 1,\dots,Q$ by
\begin{align*}
\bm u^{(j-1)}_k := \bm u_{s}(\bm{\lambda}_{j-1}^T\bm{\varphi}_{\ell_1'}(\bm \xi_k)), \enskip \bm u^{(j)}_k := \bm u_{s}(\bm{\lambda}_{j}^T\bm{\varphi}_{\ell_2'}(\bm \xi_k)), \enskip \bm u^{(j+1)}_k := \bm u_{s}(\bm{\lambda}_{j+1}^T\bm{\varphi}_{\ell_3'}(\bm \xi_k)).
\end{align*}
Thus, the finite volume update function \eqref{eq:adaptiveFVUpdate} can be written as
\begin{align}\label{eq:momentUpQuadrature}
\bm{c}_{\ell}^{\bm{\ell}'}&\left(\bm{\lambda}_{1},\bm{\lambda}_2,\bm{\lambda}_3\right)=\sum_{k=1}^Q w_k \left[\bm u^{(j)}_k- \frac{\Delta t}{\Delta x}\left(\bm g( \bm u^{(j)}_k,\bm u^{(j+1)}_k )- \bm g( \bm u^{(j-1)}_k,\bm u^{(j)}_k )\right)\right]\bm{\varphi}_{\ell}(\bm \xi_k)^T.
\end{align}
Instead of distributing the spatial mesh on the different processors, we now distribute the quadrature set, i.e. the sum in \eqref{eq:momentUpQuadrature} can be computed in parallel. Now, after having performed the dual update, the dual variables are send to all processors. With these variables, each processor computes the solution on its portion of the quadrature set and then computes its part of the sum in \eqref{eq:momentUpQuadrature} on all spacial cells. All parts from the different processors are then added together and the full time-updated moments are distributed to all processors. From here, the dual update can again be performed. The standard IPM Algorithm~\ref{alg:IPM} and One-Shot IPM Algorithm~\ref{alg:osIPM} use this parallelization strategy accordingly. Again, we point out that stochastic-Galerkin is a variant of IPM, i.e. all presented techniques for IPM can also be used for SG. The SC algorithm that we use to compare intrusive with non-intrusive methods uses a given deterministic solver as a black box. Here, we distribute the quadrature set between all processors. Note that both, SC and IPM are based on the same deterministic solver, i.e. we use the same deterministic numerical flux $\bm g$. To our best knowledge, this allows a fair comparison of the different intrusive and non-intrusive techniques. In the following section, we will study the convergence of the expectation and variance error in pseudo-time. Recording this error is straight forward with intrusive methods, however non-intrusive methods only yield expectation value and variance at the final, steady state solution. Therefore, to record the error for SC, we have implemented a collocation code, which couples all quadrature points in each time step, allowing the computation of the error in pseudo-time. Since this adds additional communication costs, we do not use the run time of this method, but instead make use of the run times from the black-box SC code. Thereby, we are able to record the convergence of expectation values and variances in pseudo-time for the non-intrusive SC method without including additional communication costs.

\section{Results}
\label{sec:results}

\subsection{2D Euler equations with a one dimensional uncertainty}
\label{sec:resultsNACA1D}
We start by quantifying the effects of an uncertain angle of attack $\phi\sim U(0.75,1.75)$ for a NACA0012 airfoil computed with different methods. The stochastic Euler equations in two dimensions are given by
\begin{align*}
\partial_t
\begin{pmatrix}
\rho \\ \rho v_1 \\ \rho v_2 \\ \rho e
\end{pmatrix}
+\partial_{x_1}
\begin{pmatrix}
\rho v_1 \\ \rho v_1^2 +p \\ \rho v_1 v_2 \\  v_1 (\rho e+p)
\end{pmatrix}
+\partial_{x_2}
\begin{pmatrix}
\rho v_2 \\ \rho v_1 v_2 \\ \rho v_2^2+p \\ v_2 (\rho e+p)
\end{pmatrix}
=\bm{0}.
\end{align*}
These equations determine the time evolution of the conserved variables $(\rho,\rho \bm v, \rho e)$, i.e. density, momentum and energy. A closure for the pressure $p$ is given by
\begin{align*}
p = (\gamma-1)\rho\left(e-\frac12(v_1^2+v_2^2)\right).
\end{align*}
Here, the heat capacity ratio $\gamma$ is chosen to be $1.4$. The spatial mesh discretizes the flow domain around the airfoil. At the airfoil boundary, we use the Euler slip condition $\bm v^T\bm n = 0$, where $\bm n$ denotes the surface normal. At a sufficiently large distance away from the airfoil, we assume a far field flow with a given Mach number $Ma = 0.8$, pressure $p = 101\;325$ Pa and a temperature of $273.15$ K. Now the angle of attack $\phi$ is uniformly distributed in the interval of $[0.75,1.75]$ degrees, i.e. we choose $\phi(\xi) = 1.25 + 0.5\xi$ where $\xi\sim U(-1,1)$. As commonly done, the initial condition is equal to the far field boundary values. Consequently, the wall condition at the airfoil is violated initially and will correct the flow solution. 

The computational domain is a circle with a diameter of $40$ meters. In the center, the NACA0012 airfoil with a length of one meter is located. The spatial mesh is composed of a coarsely discretized far field and a finely resolved region around the airfoil, since we are interested in the flow solution at the airfoil. Altogether, the mesh consists of 22361 triangular elements.

\begin{figure}[h!]
\centering
	\begin{subfigure}{0.329\linewidth}
		\centering
				\includegraphics[width=\linewidth]{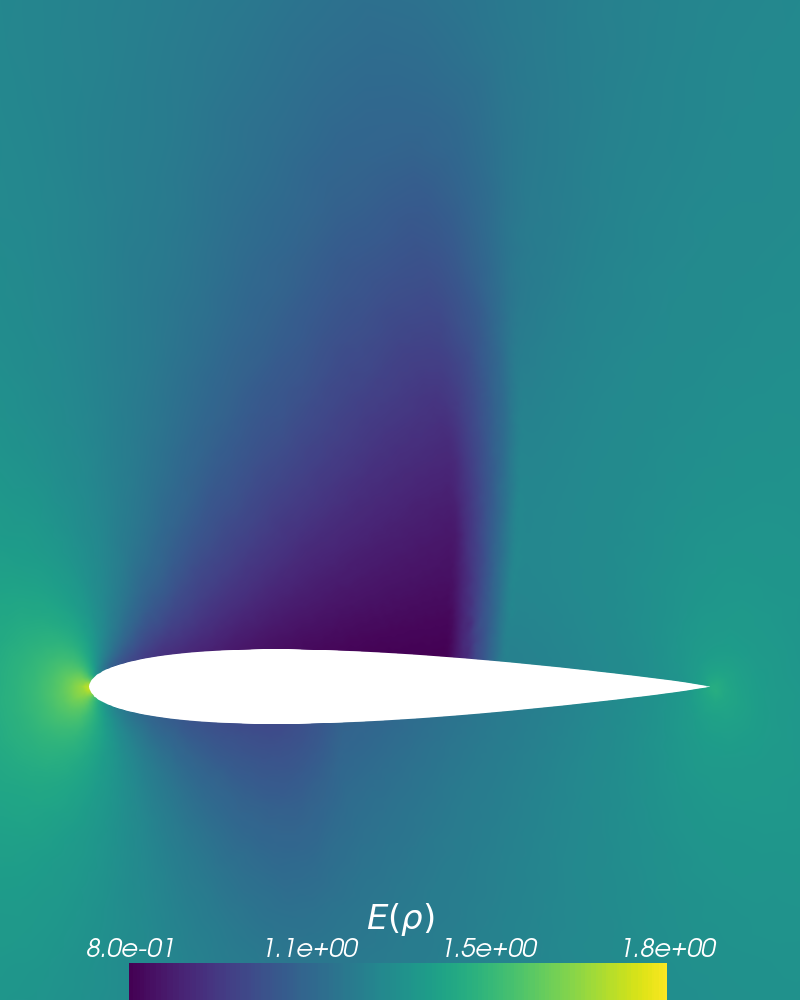}
		\label{fig:referenceSolutionsub1}
	\end{subfigure}
	\hfill
	\begin{subfigure}{0.329\linewidth}
		\centering
				\includegraphics[width=\linewidth]{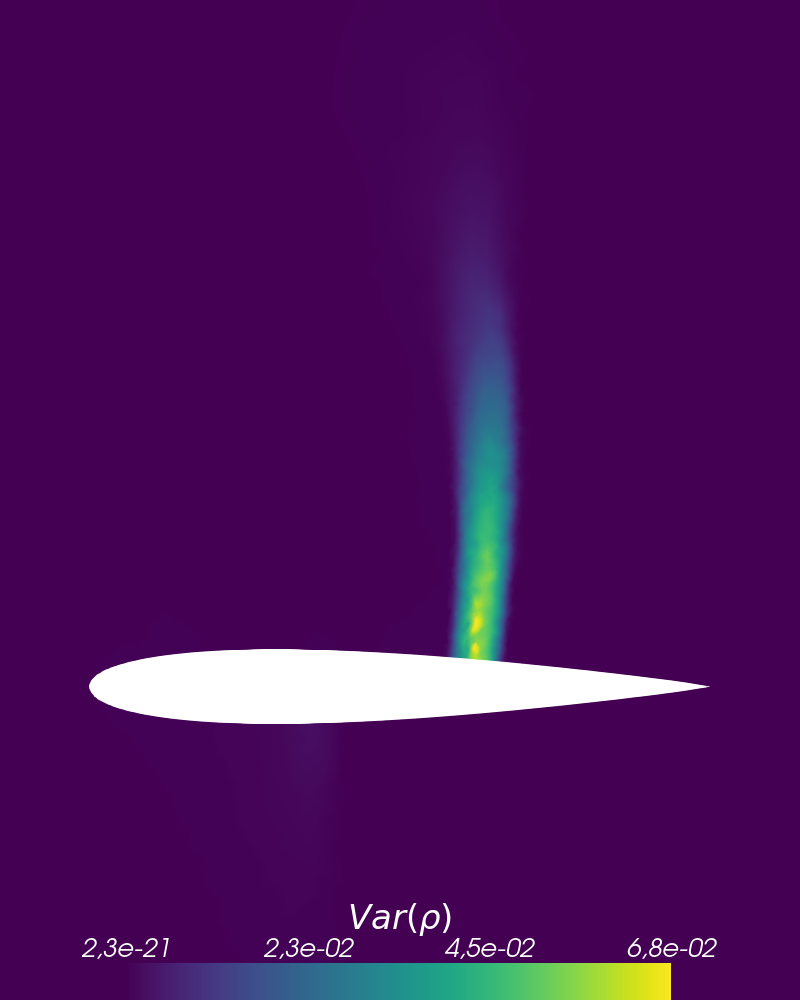}
		\label{fig:referenceSolutionsub2}
	\end{subfigure}
	\hfill
	\begin{subfigure}{0.329\linewidth}
		\centering
				\includegraphics[width=\linewidth]{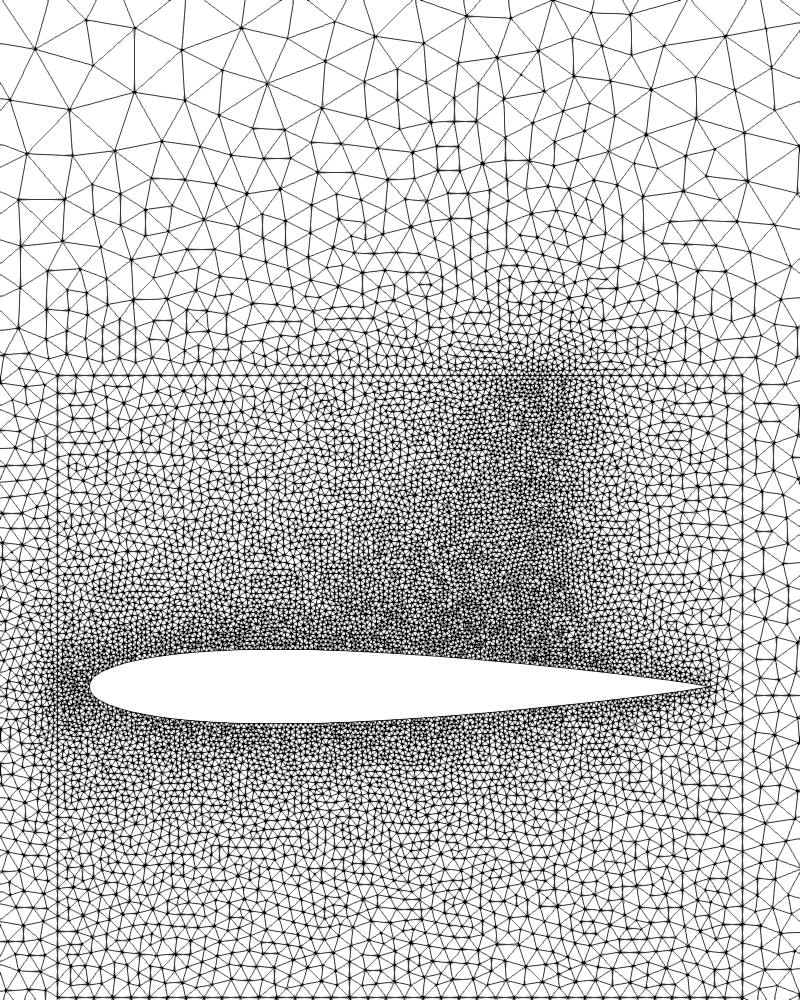}
		\label{fig:referenceSolutionsMesh}
	\end{subfigure}
	\caption{Reference solution E$[\rho]$ and Var$[\rho]$ and the mesh close to the airfoil which is used in the computation of all presented methods.}
	\label{fig:referenceSolution}
\end{figure}

The aim is to quantify the effects arising from the one-dimensional uncertainty $\xi$ and to investigate its effects on the solution with different methods. To be able to measure the quality of the obtained solutions, we compute a reference solution using stochastic-Collocation with $100$ Gauss-Legendre quadrature points, which can be found in Figure~\ref{fig:referenceSolution}. In the following, we investigate the L$^2$-error of the variance and the expectation value. The L$^2$-error of the discrete quantity $\bm e_{\Delta}=(\bm e_1,\cdots,\bm e_{N_x})^T$, where $\bm e_j$ is the cell average of the quantity $\bm e$ in spatial cell $j$, is denoted by
\begin{align*}
\Vert \bm e_{\Delta} \Vert_{\Delta} := \sqrt{\sum_{j=1}^{N_x} \Delta x_j \bm e_j^2}\;.
\end{align*}
Hence, when denoting the reference solution by $\bm u_{\Delta}$ and the moments obtained with the numerical method by $\bm{\hat u}_{\Delta}$, we investigate the relative error
\begin{align*}
\frac{\Vert \text{E}[\bm u_{\Delta}] - \text{E}[\mathcal{U}(\bm{\hat u}_{\Delta})] \Vert_{\Delta}}{\Vert \text{E}[\bm u_{\Delta}] \Vert_{\Delta}} \qquad \text{ and }\qquad \frac{\Vert \text{Var}[\bm u_{\Delta}] - \text{Var}[\mathcal{U}(\bm{\hat u}_{\Delta})] \Vert_{\Delta}}{\Vert \text{Var}[\bm u_{\Delta}] \Vert_{\Delta}}.
\end{align*}
The error is computed inside a box of one meter height and 1.1 meters length around the airfoil to prevent small fluctuations in the coarsely discretized far field from affecting the error.

The quantities of interests are now computed with the different methods. All methods in this section have been computed using five MPI threads. For more information on the chosen entropy and the resulting solution ansatz for IPM, see \ref{app:IPM2DEuler}. 

\begin{figure}[h!]
\centering
\centering
	\begin{subfigure}{0.5\linewidth}
		\centering
				\includegraphics[width=\linewidth]{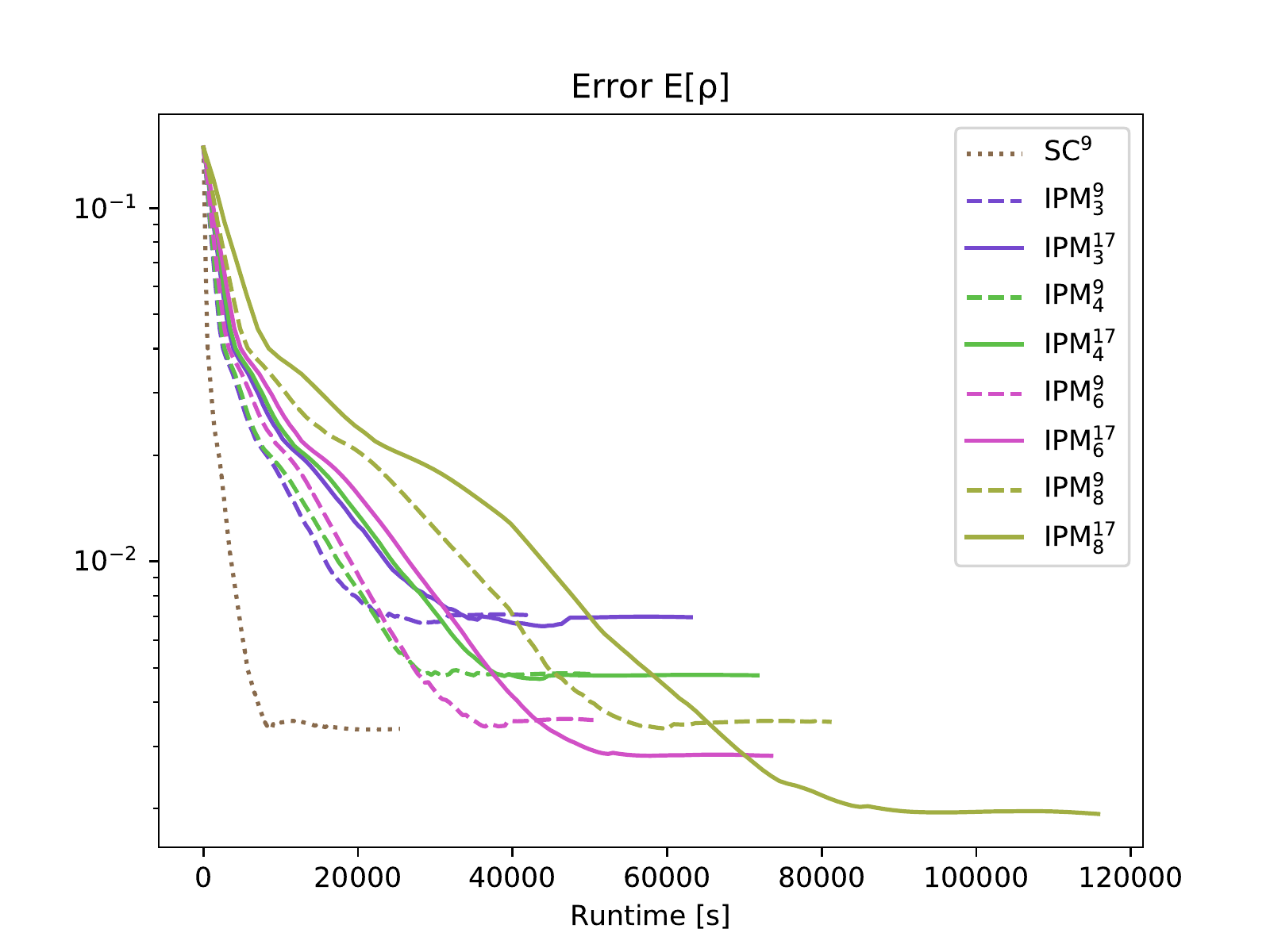}
		\caption{}
		\label{fig:ErrorDifferentQuadA}
	\end{subfigure}%
	\begin{subfigure}{0.5\linewidth}
		\centering
				\includegraphics[width=\linewidth]{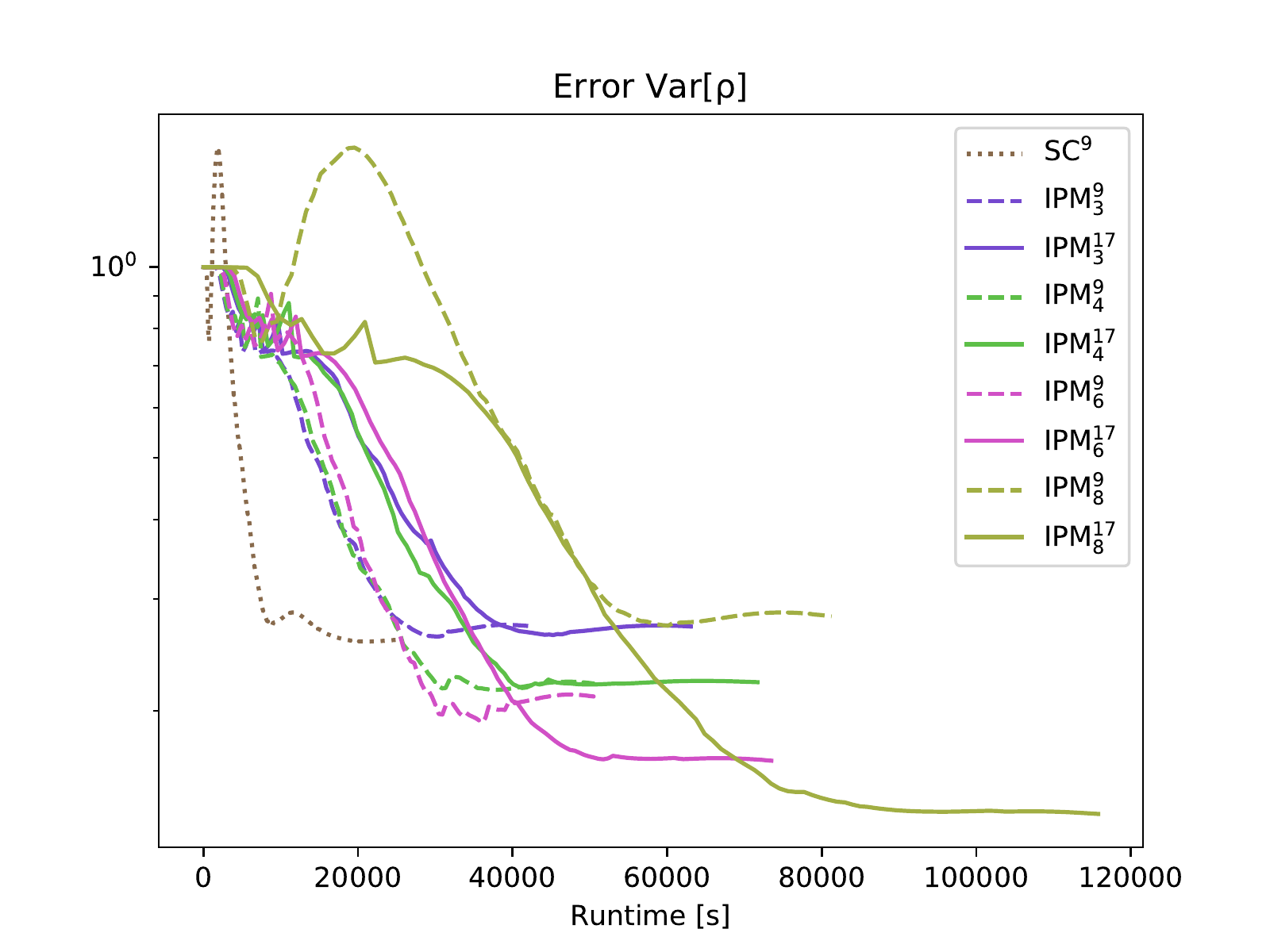}
		\caption{}
		\label{fig:ErrorDifferentQuadB}
	\end{subfigure}
	\caption{Relative L$^2$-error with different quadrature levels for IPM in comparison with SC. The subscript denotes the moment order, the superscript denotes the number of Clenshaw-Curtis quadrature points.}
	\label{fig:ErrorDifferentQuad}
\end{figure}
\begin{figure}[h!]
\centering
	\begin{subfigure}{0.5\linewidth}
		\centering
				\includegraphics[width=\linewidth]{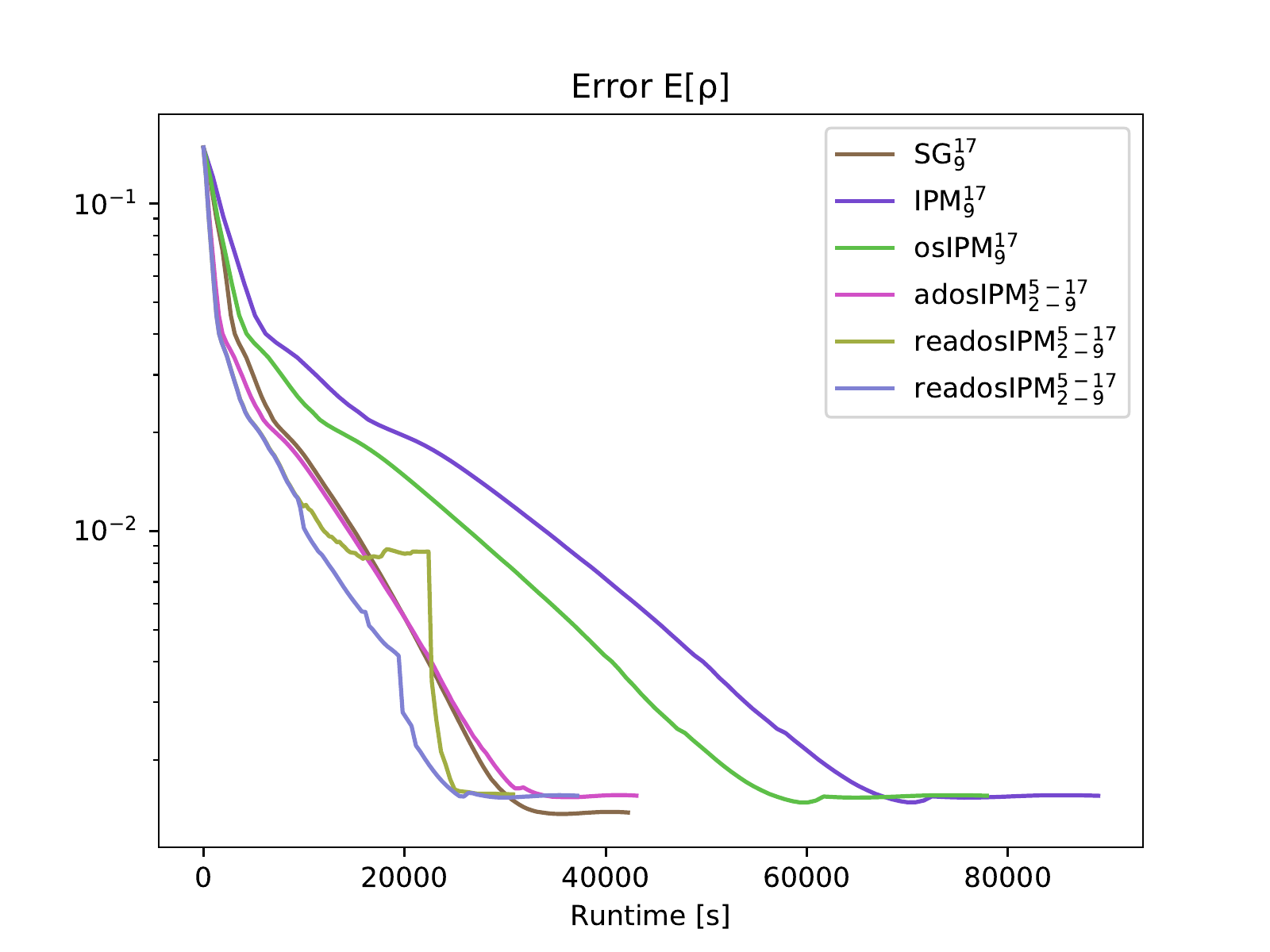}
		\caption{}
		\label{fig:sub31}
	\end{subfigure}%
	\begin{subfigure}{0.5\linewidth}
		\centering
				\includegraphics[width=\linewidth]{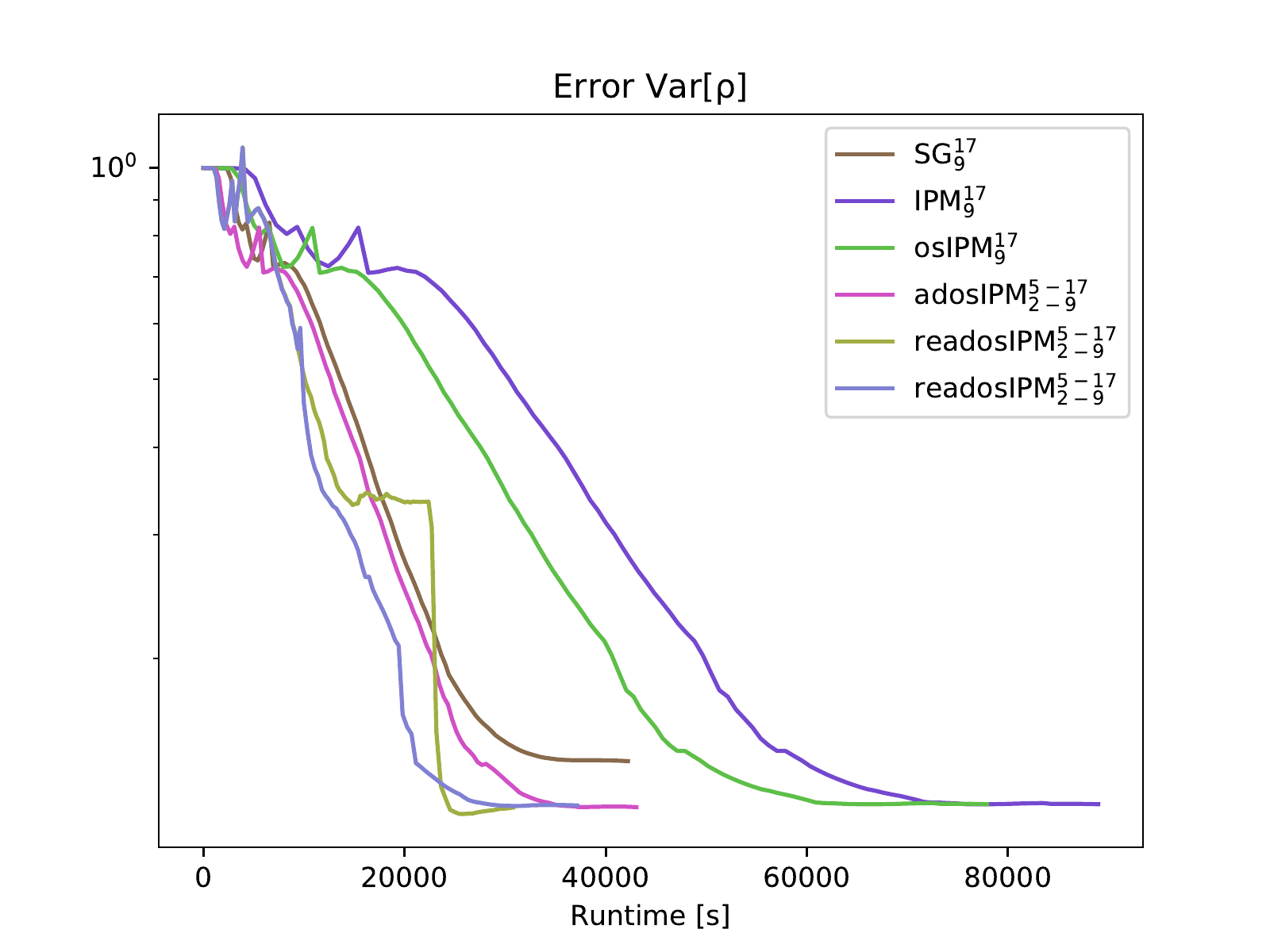}
		\caption{}
		\label{fig:sub32}
	\end{subfigure}
	
	\begin{subfigure}{0.5\linewidth}
		\centering
				\includegraphics[width=\linewidth]{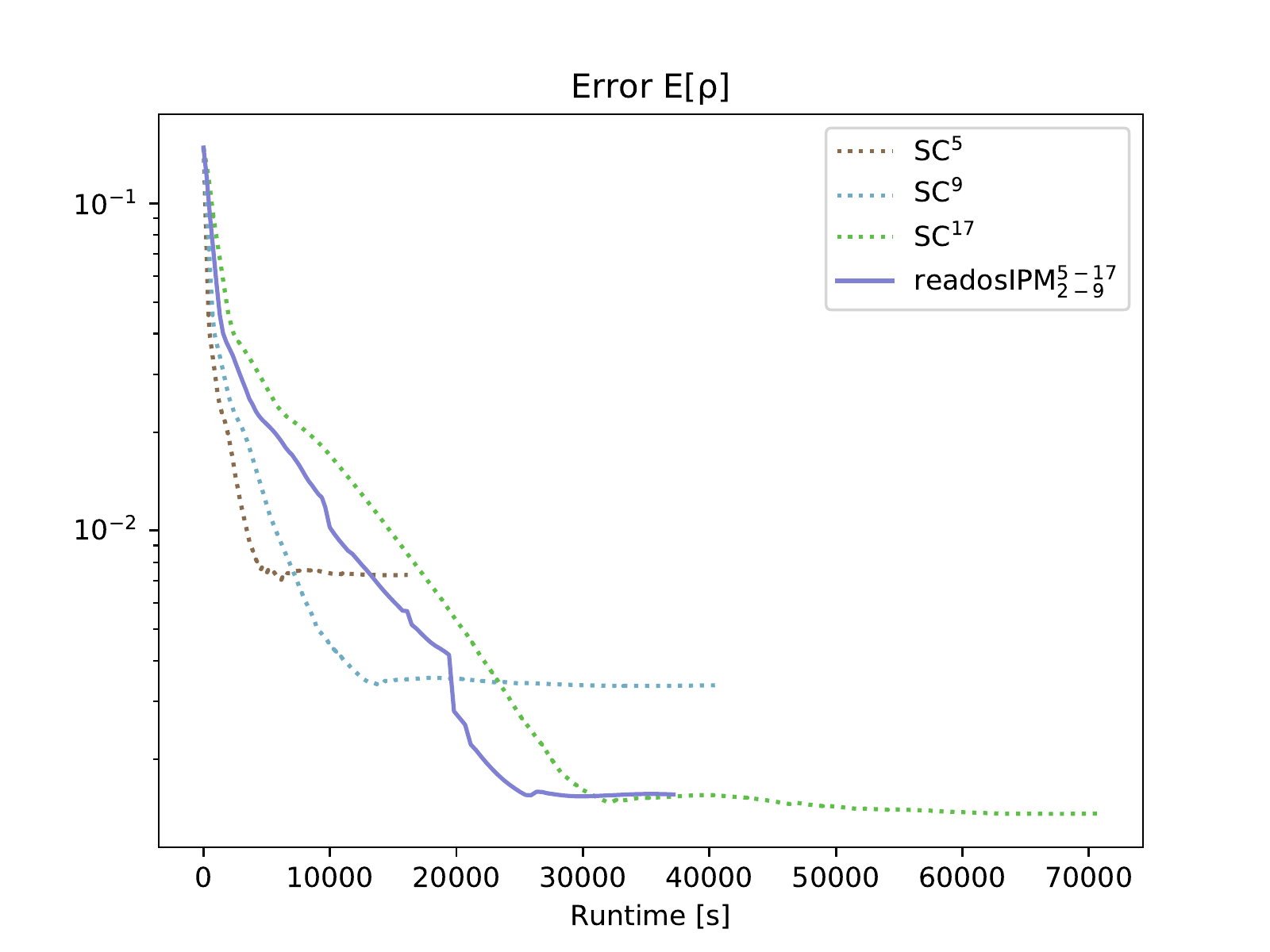}
		\caption{}
		\label{fig:sub33}
	\end{subfigure}%
	\begin{subfigure}{0.5\linewidth}
		\centering
				\includegraphics[width=\linewidth]{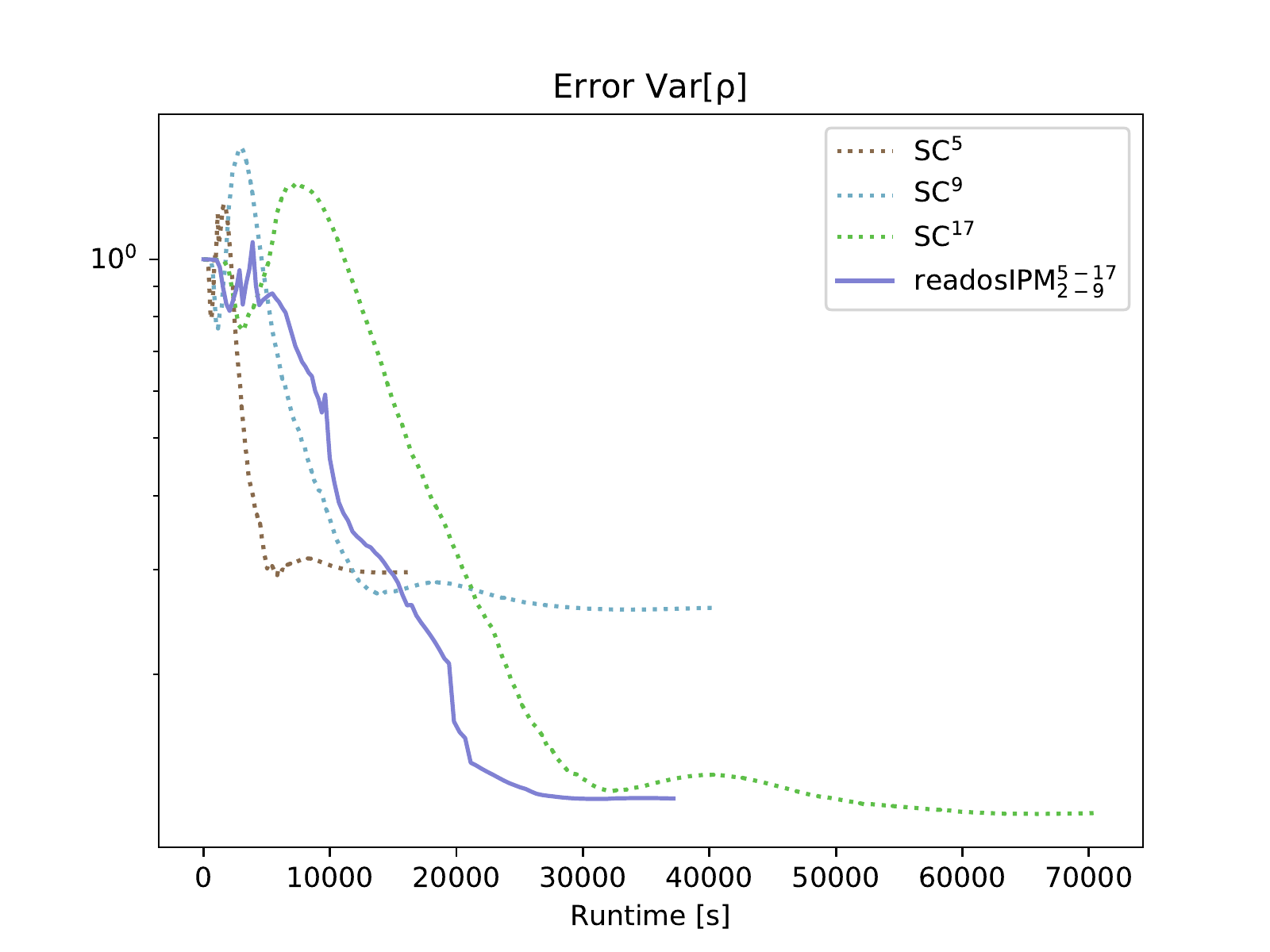}
		\caption{}
		\label{fig:sub34}
	\end{subfigure}
	\caption{Comparison of the relative L$^2$-error for the density for IPM related methods (first row) and of the best performing IPM method in comparison with SC (second row). All intrusive methods converge to a residual $\varepsilon=6\cdot 10^{-6}$, whereas all non-intrusive methods converge to a residual $\varepsilon=1\cdot 10^{-7}$. All computations are performed with 5 MPI threads.}
	\label{fig:L2ErrorSolution}
\end{figure}
\begin{figure}[h!]
\centering
	\begin{subfigure}{0.329\linewidth}
		\centering
		\includegraphics[width=\linewidth]{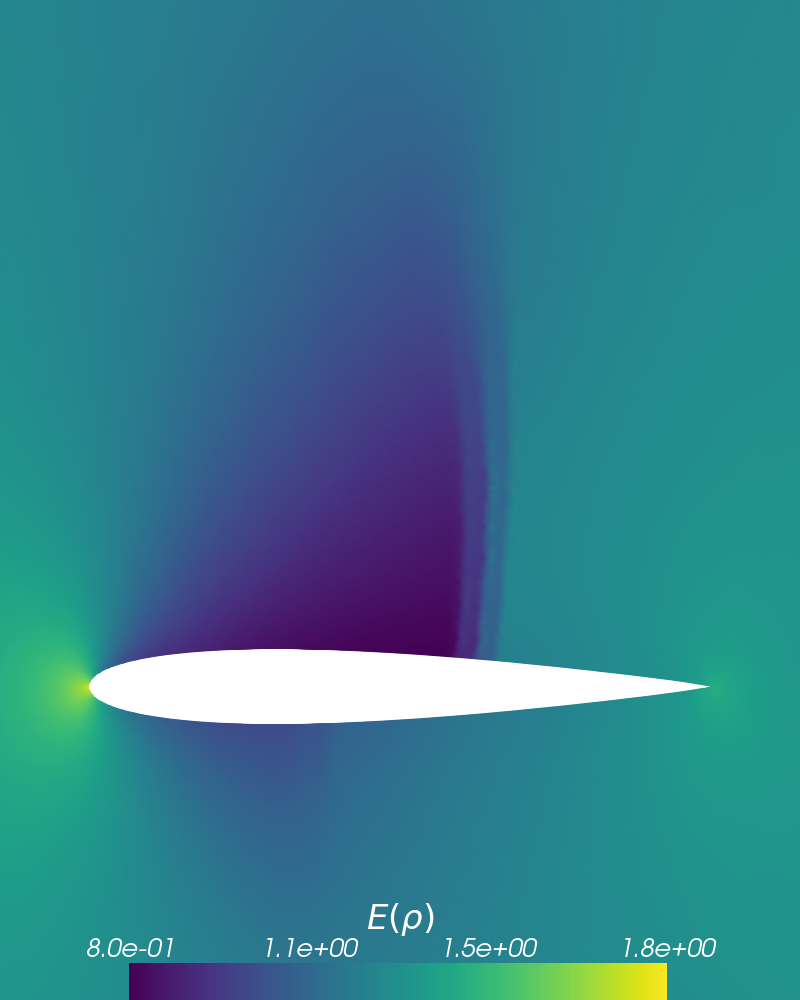}
		\label{fig:sub1}
	\end{subfigure}%
	\hfill
	\begin{subfigure}{0.329\linewidth}
		\centering
		\includegraphics[width=\linewidth]{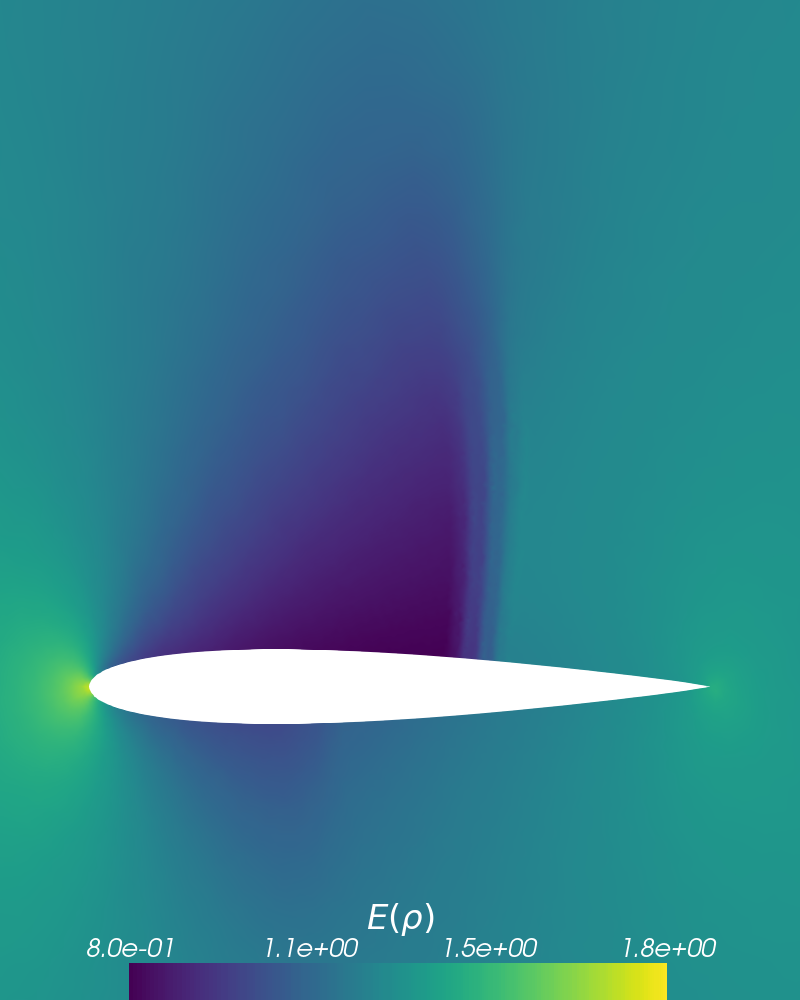}
		\label{fig:sub2}
	\end{subfigure}%
	\hfill
	\begin{subfigure}{0.329\linewidth}
		\centering
		\includegraphics[width=\linewidth]{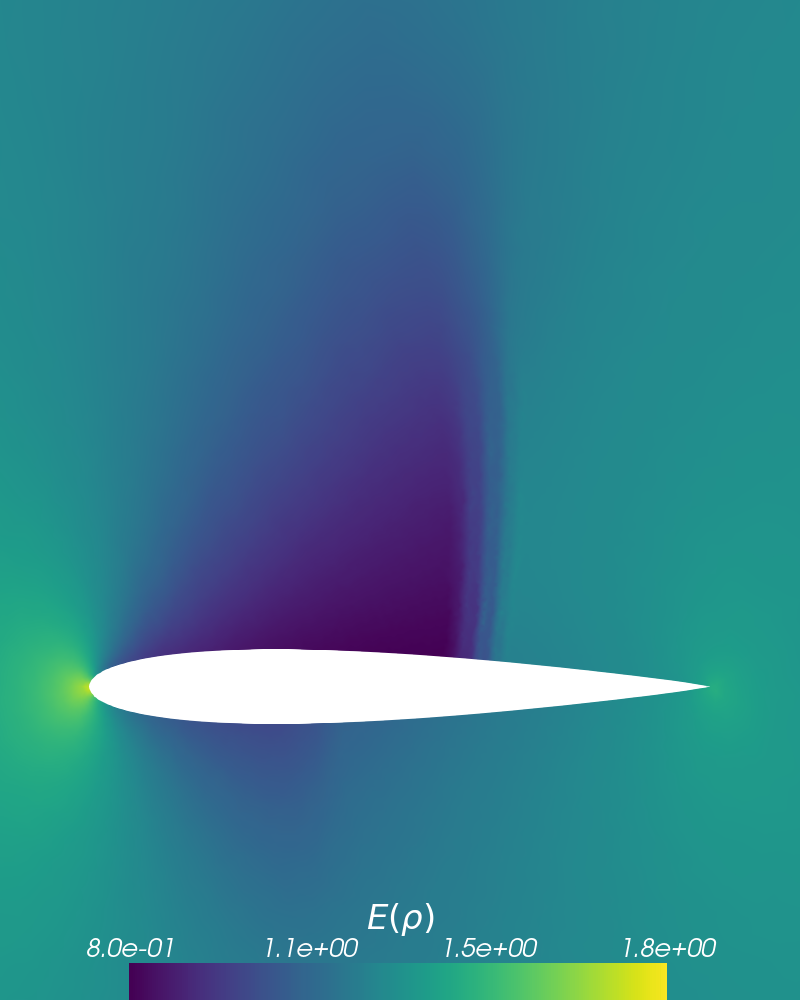}
		\label{fig:sub1}
	\end{subfigure}\\
	\vspace{-0.35cm}
	\begin{subfigure}{0.329\linewidth}
		\centering
		\includegraphics[width=\linewidth]{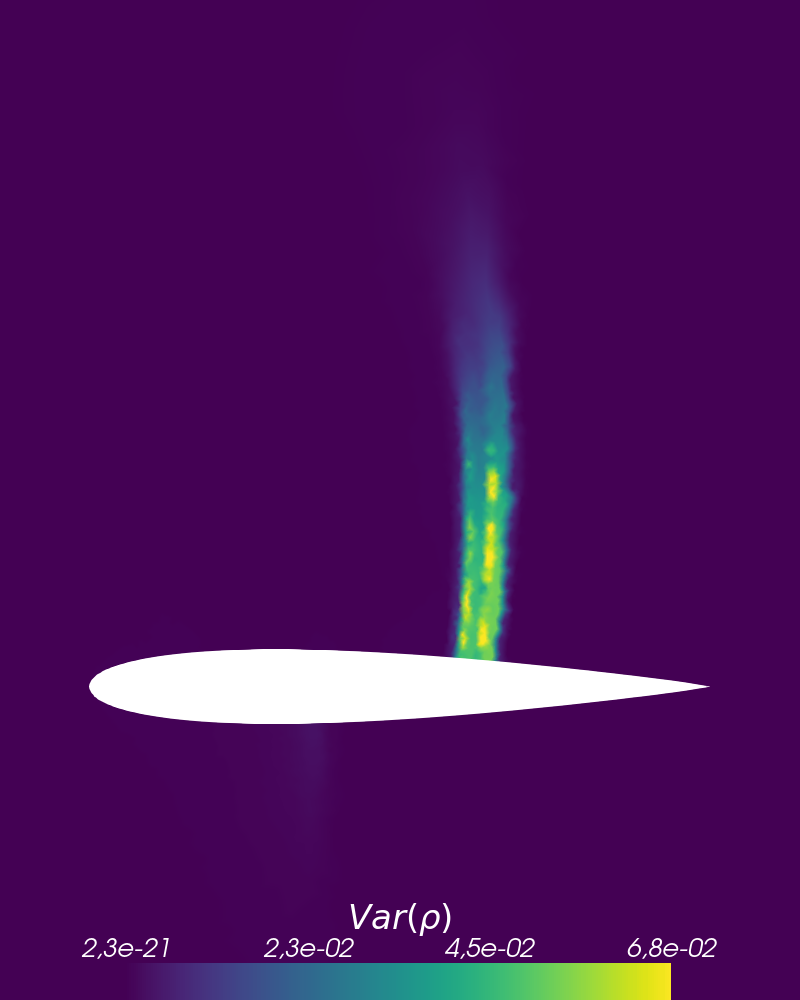}
		\label{fig:sub1}
	\end{subfigure}%
	\hfill
	\begin{subfigure}{0.329\linewidth}
		\centering
		\includegraphics[width=\linewidth]{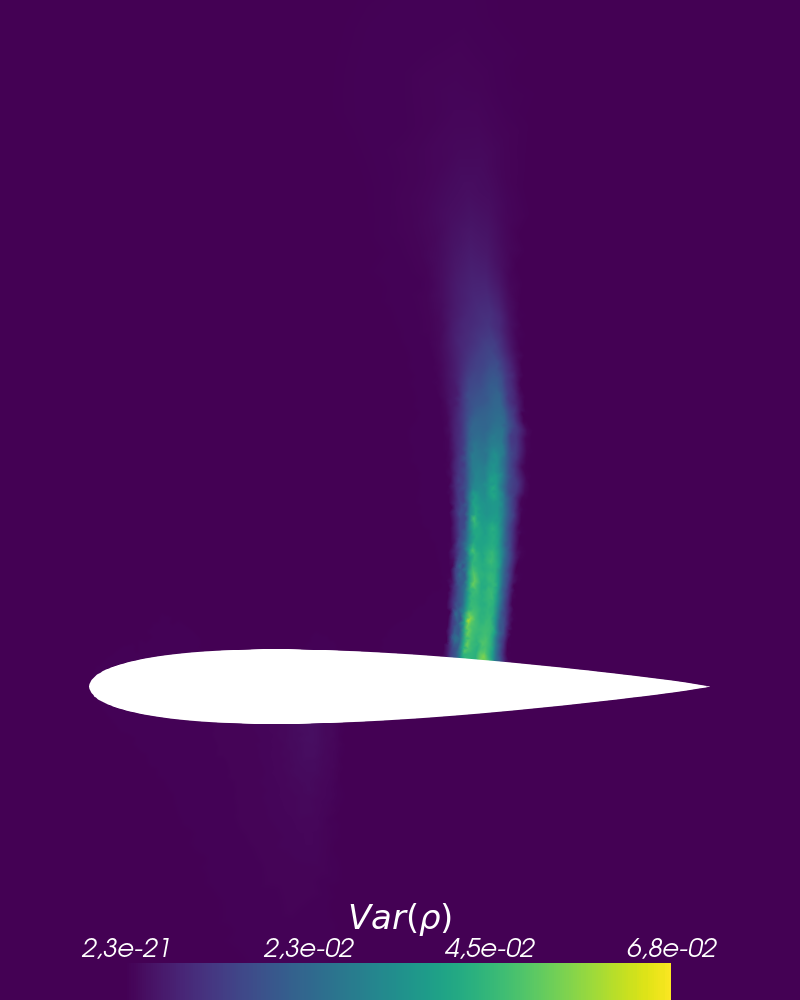}
		\label{fig:sub2}
	\end{subfigure}%
	\hfill
	\begin{subfigure}{0.329\linewidth}
		\centering
		\includegraphics[width=\linewidth]{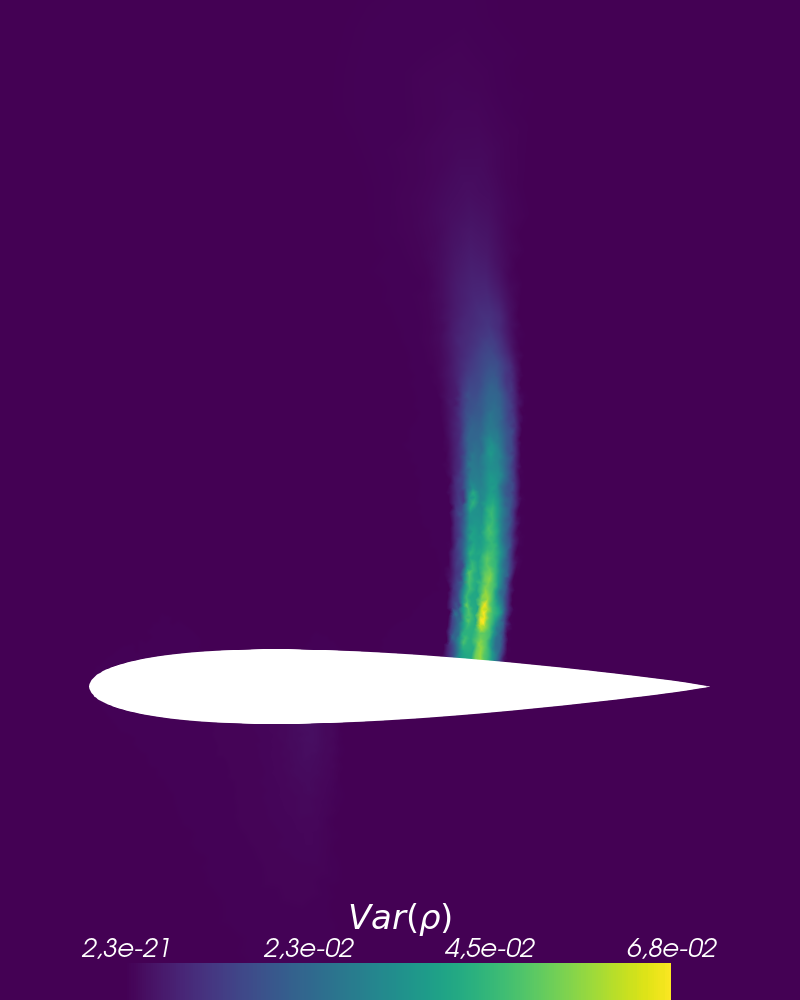}
		\label{fig:sub1}
	\end{subfigure}
	\caption{E$[\rho]$ and Var$[\rho]$ computed with SC$^5$, SG$_4$, IPM$_4$ (from left to right). Compare to the reference solution shown in Figure \ref{fig:referenceSolution}.}
	\label{fig:ERhoVarRho5}
\end{figure}

Recall that the numerical flux \eqref{eq:numericalFluxIPM} uses a quadrature rule to approximate integrals. We start by investigating the effects this quadrature has on the solution accuracy. For this, we run the IPM method with a moment order ranging from 3 to 7 using a Clenshaw-Curtis quadrature rule with level three (i.e. 9 quadrature points) and level four (i.e. 17 quadrature points). A comparison of the error obtained with these two quadrature levels is given in Figure~\ref{fig:ErrorDifferentQuad}. To denote the number of quadrature points, we use a superscript and the moment order is denoted by a subscript. We observe the following:
\begin{itemize}
\item When for example comparing the error obtained with IPM$_8^9$ and IPM$_8^{17}$ (i.e. the polynomial order is $8$, meaning that $9$ moments are used and the computation is done using $9$ and $17$ quadrature points) it can be seen that the error stagnates when the chosen quadrature is not sufficiently accurate. Hence, the accuracy level is dominated by aliasing effects that result from an inaccurate quadrature rule in the numerical flux and not the truncation error of the moment system.
\item If the truncation order is sufficiently small, both quadrature levels yield the same accuracy. We observe this behavior for the expectation value until a truncation order of $N=5$ and for the variance for $N=4$. Hence, the variance is more sensitive to aliasing errors.
\item Figure~\ref{fig:ErrorDifferentQuadA} reveals that the IPM error of E[$\rho$] with 9 quadrature points stagnates at the error level of SC$^9$, i.e. the aliasing error heavily affects the accuracy. This behavior becomes more dominant when looking at the variance error in Figure~\ref{fig:ErrorDifferentQuadB}. Here, IPM$_8^{17}$ yields a significantly improved result compared to IPM$_8^9$. Again, the IPM$_8^9$ result stagnates at the SC$^9$ accuracy level, which can be reached with IPM when only using four moments (combined with a sufficiently accurate quadrature level). Hence, the number of moments needed for IPM to obtain a certain variance error is significantly smaller than the number of quadrature points needed for SC. This result can be observed throughout the numerical experiments of this work. Especially for high dimensional problems, this potentially decreases the number of unknowns to reach a certain accuracy level significantly. 
\end{itemize}

Let us now compare results obtained with stochastic-Galerkin and IPM as well as its proposed acceleration techniques at a fixed moment order $9$. Note, that since IPM generalizes SG, all proposed techniques can be used for stochastic-Galerkin as well. All adaptive methods use gPC polynomials of order 2 to 9 (i.e. 3 to 10 moments). Order $2$ uses $5$ quadrature points, orders $3$ to $6$ use $9$ quadrature points and orders $8$ and $9$ use $17$ quadrature points. When the smoothness indicator \eqref{eq:errorIndicator} lies below $\delta_{-} = 2\cdot 10^{-4}$, the adaptive methods decrease the truncation order, if it lies above $\delta_{+} = 2\cdot 10^{-5}$ the truncation order is increased. The remaining methods have been computed with $17$ quadrature points. The iteration in pseudo-time is performed until the expectation value of the density fulfills the stopping criterion \eqref{eq:residualSteady} with $\varepsilon = 6\cdot 10^{-6}$, however it can be seen that the error saturates already at a bigger residual.

First, let us mention that the adaptive SG method fails, since it yields negative densities during the iteration. The standard SG method however preserves positivity of mass, energy and pressure. The change of the relative L$^2$-error during the iteration to the steady state has been recorded in Figure~\ref{fig:sub31} for the expectation value and in Figure~\ref{fig:sub32} for the variance. When comparing intrusive methods without acceleration techniques as well as SC, the following properties emerge:
\begin{itemize}
\item Compared to IPM, stochastic-Galerkin comes at a significantly reduced runtime, meaning that the IPM optimization problem requires a significant computational effort.
\item For the expectation value SG$_9$ shows a smaller error compared to IPM$_9$, while for the variance, we see the opposite, i.e. IPM yields a better solution approximation than SG.
\end{itemize}
The proposed acceleration techniques show the following behavior:
\begin{itemize}
\item The One-Shot IPM (osIPM) method proposed in Section~\ref{sec:OneShotIPM} reduces the runtime while yielding the same error as the classical IPM method.
\item When using adaptivity (see Algorithm~\ref{alg:ad-IPM}) in combination with the One-Shot idea, the method is denoted by \textit{adaptive One-Shot IPM} (adosIPM). This method reaches the steady state IPM solution with a faster runtime than SG.
\item The idea of refinement retardation combined with adosIPM (see Algorithm~\ref{alg:adosIPM}) is denoted by \textit{retardation adosIPM} (readosIPM), which further decreases runtime. Here, we use two different strategies to increase the accuracy: First, we steadily increase the maximal truncation order when the residual approaches zero. To determine residual values for a given set of truncation orders $2,4,5$ and $8$, we study at which residual level the IPM method reaches a saturated error level for each truncation order. The residual values are then determined to be $6\cdot 10^{-5},3\cdot 10^{-5},2.2\cdot 10^{-5}$ and $2\cdot 10^{-5}$. The second, straight forward strategy converges the solution on a low truncation order of $2$ to a residual of $10^{-5}$ and then switches to a maximal truncation order of $9$. Strategy 1 is depicted in purple, Strategy 2 is depicted in yellow. It can be seen that both approaches reach the IPM$_9$ error for the same run time. Hence, we deduce that a naive choice of the refinement retardation strategy suffices to yield a satisfactory behavior.
\end{itemize}

Let us now compare the results from intrusive methods with those of SC. For every quadrature point, SC iterates the solution until the density lies below a threshold of $\varepsilon = 1\cdot 10^{-7}$. Recording the error of intrusive methods during the iteration is straight forward. To record the error of SC, we couple all quadrature points after each iteration to evaluate the expectation value and variance, which destroys the black-box nature and results in additional costs. The collocation runtimes, that are depicted on the $x$-axis in Figures~\ref{fig:ErrorDifferentQuad}, \ref{fig:sub33}, \ref{fig:sub34} and \ref{fig:L2ErrorSolution2D}, are however rescaled to the runtimes that we achieve when running the collocation code in its original, black-box framework without recording the error. Our stochastic-Collocation computations use Clenshaw-Curtis quadrature levels $2,3$ and $4$, i.e. $5,9$ and $17$ quadrature points. The comparison of intrusive methods with non-intrusive methods shows the following:
\begin{itemize}
\item SC requires a smaller residual to converge to a steady state solution (we converge the results to $\varepsilon = 1\cdot 10^{-7}$ compared to $\varepsilon = 6\cdot 10^{-6}$ for intrusive methods).
\item Again, intrusive methods yield improved solutions compared to SC with the same number of unknowns. Actually, the error obtained with 17 unknowns when using SC is comparable with the error obtained with 10 unknowns when using intrusive methods.
\end{itemize}

Let us finally take a look at the expectation value and variance computed with different methods. All results are depicted for a zoomed view around the airfoil. Figure~\ref{fig:ERhoVarRho5} shows the expectation value (first row) and variance (second row) computed with $5$ quadrature points for SC and $5$ moments for SG and IPM. One can observe the following
\begin{itemize}
\item All methods yield non-physical step-like profiles of the expectation value and variance along the airfoil. This effect can be observed in various settings \cite{le2004uncertainty,kusch2018filtered,poette2019contribution,barth2013non,dwight2013adaptive} and stems from Gibb's phenomena in the polynomial description of the uncertainty, either in the gPC polynomials of intrusive methods or the Lagrange polynomials used in collocation techniques. 
\item The jump position of the intrusive solution profiles capture the exact behavior more accurately.
\end{itemize}
The readosIPM results as well as the corresponding refinement levels are depicted in Figure~\ref{fig:readosIPMEVar}. One observes that the solution no longer shows the previously observed discontinuous profile and yield a satisfactory agreement with the reference solution depicted in Figure~\ref{fig:referenceSolution}. The refinement level shows that away from the airfoil, a refinement level of 0 (i.e. a truncation order of 2) suffices to yield the IPM$_9$ solution. The region with a high variance requires a refinement level of 7 (truncation order 9).
\begin{figure}[h!]
\centering
	\begin{subfigure}{0.329\linewidth}
		\centering
		\includegraphics[width=\linewidth]{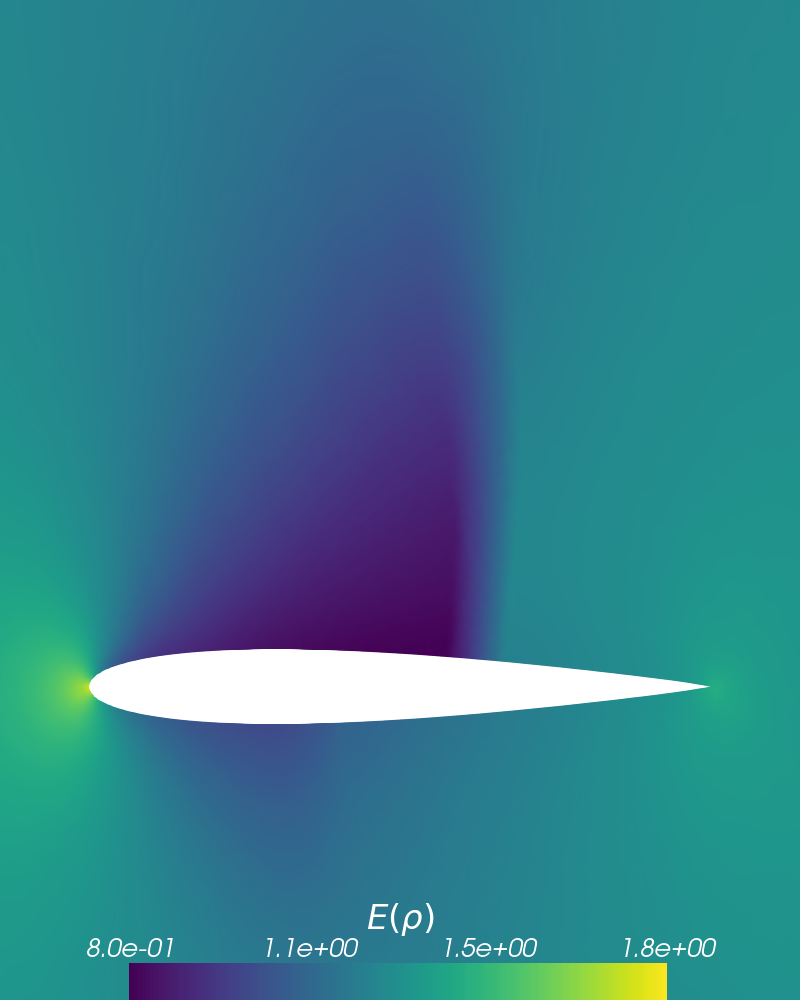}
		\label{fig:sub1}
	\end{subfigure}%
	\hfill
	\begin{subfigure}{0.329\linewidth}
		\centering
		\includegraphics[width=\linewidth]{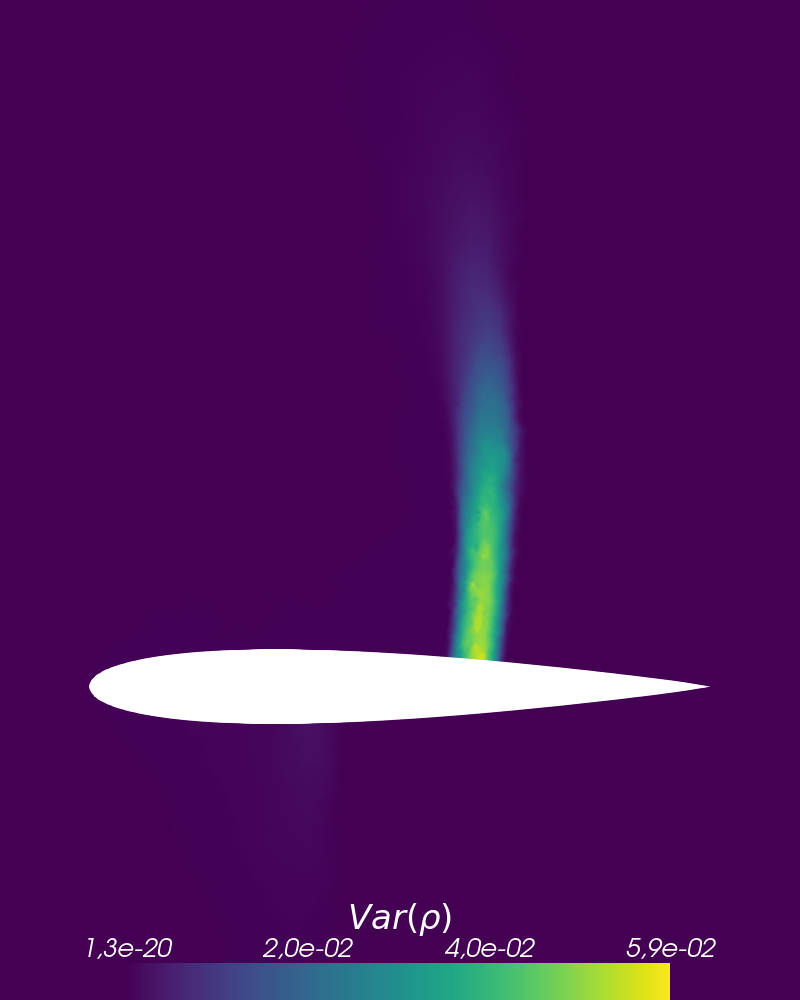}
		\label{fig:sub2}
	\end{subfigure}%
	\hfill
	\begin{subfigure}{0.329\linewidth}
		\centering
		\includegraphics[width=\linewidth]{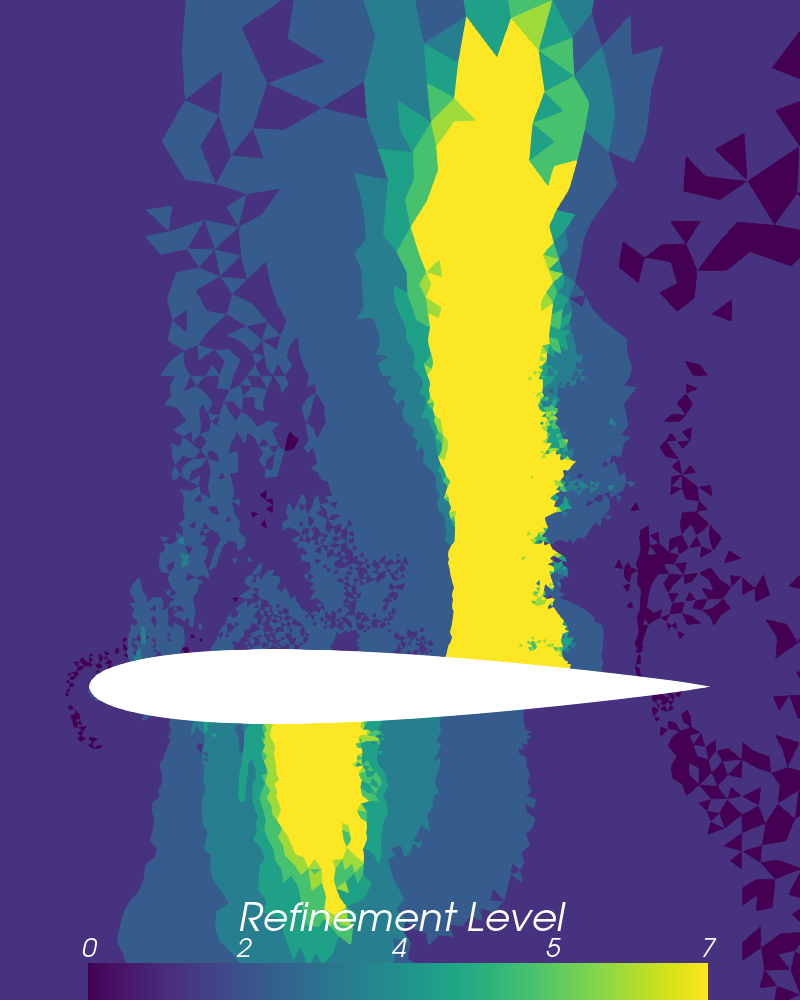}
		\label{fig:sub1}
	\end{subfigure}
	\caption{E$[\rho]$, Var$[\rho]$ and refinement level for readosIPM$_{2-9}$.}
	\label{fig:readosIPMEVar}
\end{figure}


\subsection{2D Euler equations with a two dimensional uncertainty}

In the following, we assume a fixed angle of attack with $\phi = 1.25$ degrees and study the effect of two sources of uncertainties, namely the farfield pressure and Mach number. The farfield pressure is ${p \sim U(100\;325,102\;325)}$ Pa and the Mach number is $Ma \sim U(0.775,0.825)$. Since this problem only has a two-dimensional uncertainty, we use a tensorized quadrature set, which in our experiments proved to be more efficient than a sparse grid quadrature. For SC, we use quadrature sets with $5^2$, $9^2$ and $17^2$ quadrature nodes and compare against SG with moments up to total degree $9$ as well as adaptive IPM with refinement retardation (with and without the One-Shot strategy). The IPM method uses moment orders ranging from $1$ to $9$ adaptively with refinement barriers $\delta_{-} = 1\cdot 10^{-4}$ and $\delta_{+} = 1\cdot 10^{-5}$. The refinement retardation allows the truncation order to have a maximal total degree of $1$ until a residual of $\varepsilon = 1.5\cdot 10^{-5}$ and then increases the truncation order by one when the residual is reduced by an amount of $5\cdot 10^{-6}$. Hence, the maximal truncation order of $9$ is reached when the residual is below $\varepsilon = 7\cdot 10^{-6}$. Again, the refinement strategy let to negative densities for SG resulting in a failure of the method.
\begin{figure}[h!]
\centering
	\begin{subfigure}{0.5\linewidth}
		\centering
				\includegraphics[width=\linewidth]{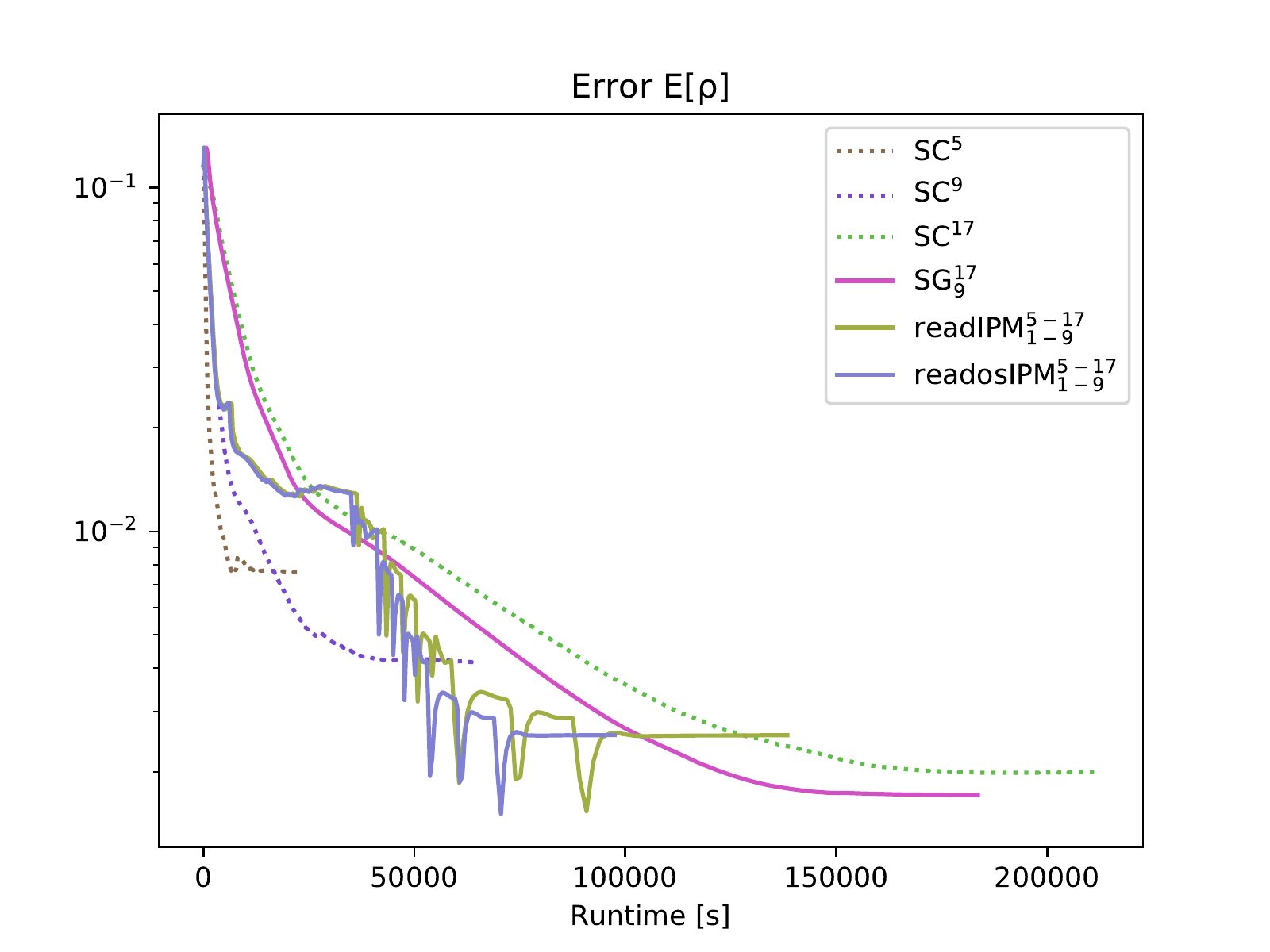}
		\label{fig:sub1}
	\end{subfigure}%
	\begin{subfigure}{0.5\linewidth}
		\centering
				\includegraphics[width=\linewidth]{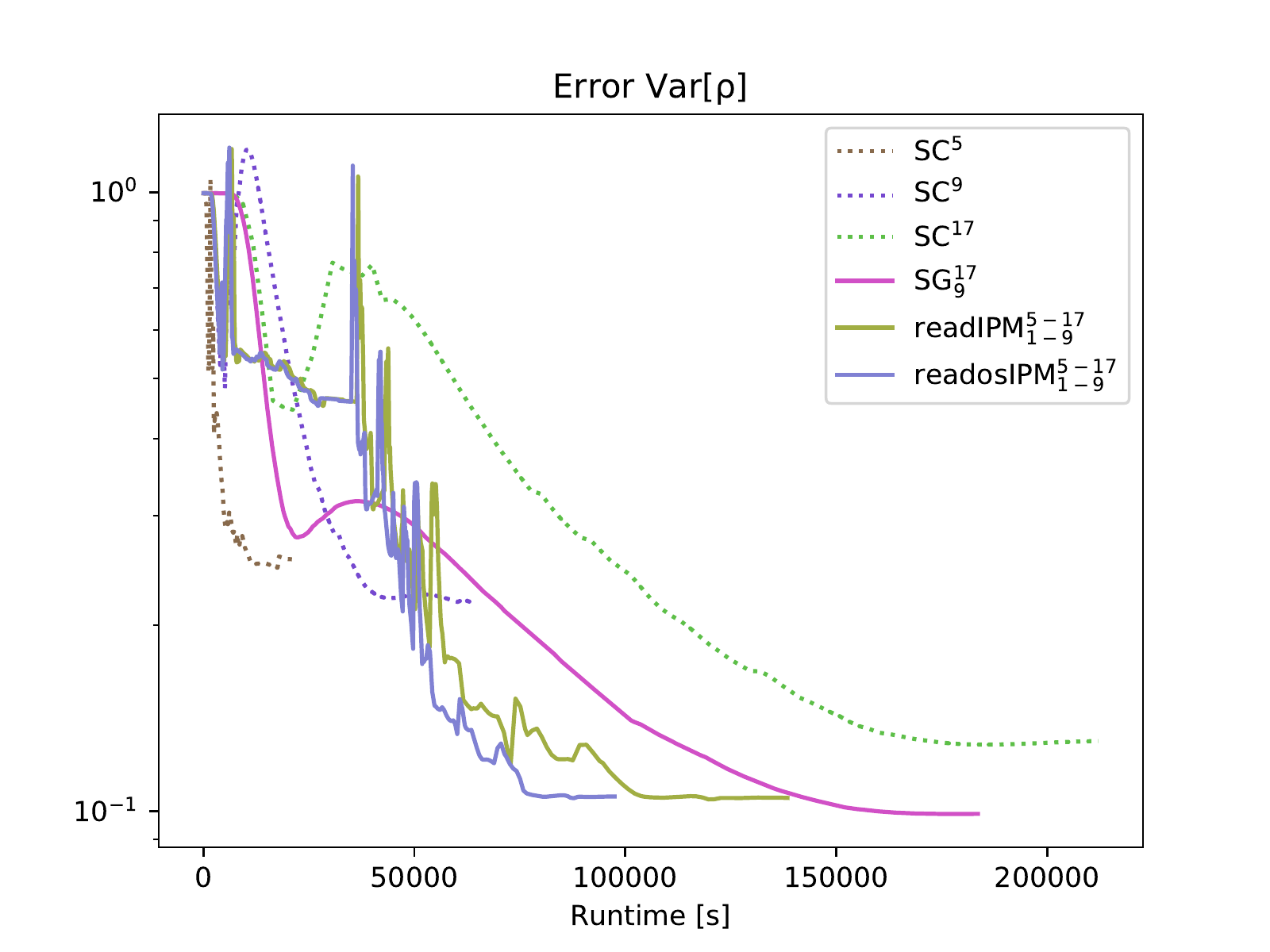}
		\label{fig:sub2}
	\end{subfigure}
	
	\begin{subfigure}{0.5\linewidth}
		\centering
				\includegraphics[width=\linewidth]{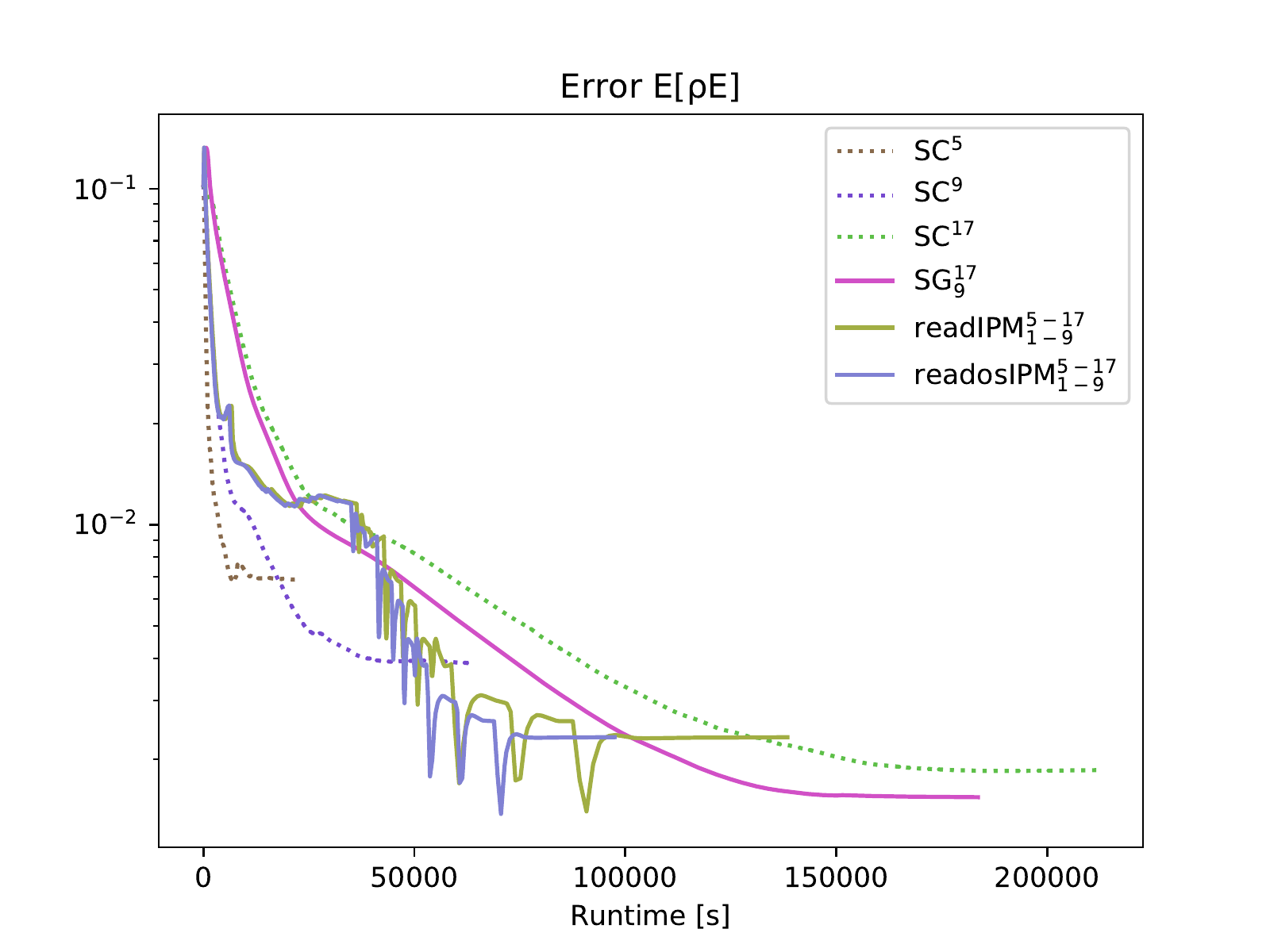}
		\label{fig:sub1}
	\end{subfigure}%
	\begin{subfigure}{0.5\linewidth}
		\centering
				\includegraphics[width=\linewidth]{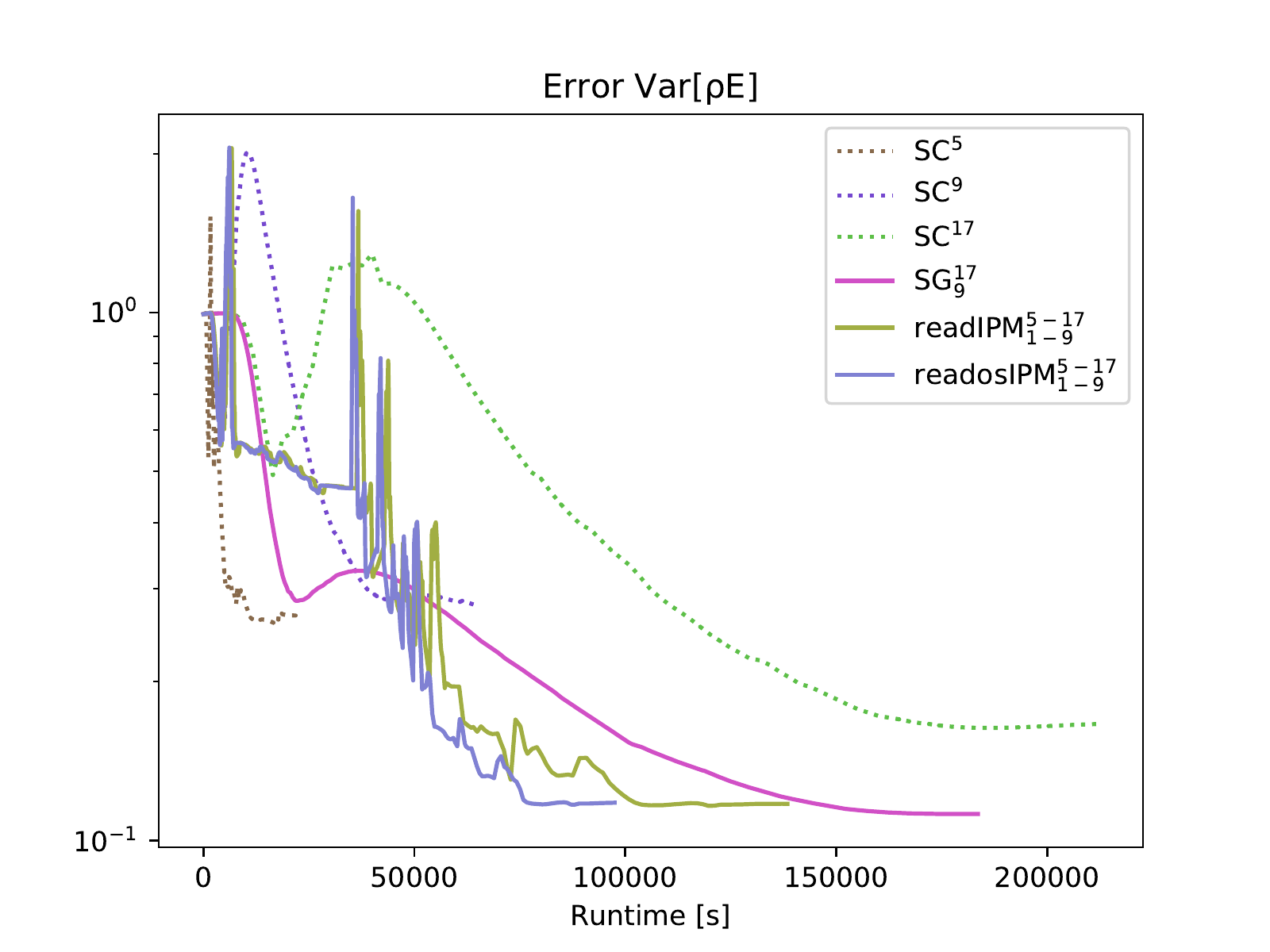}
		\label{fig:sub2}
	\end{subfigure}
	\caption{Relative L$^2$-error with 16 MPI threads for density and energy. All intrusive methods are converged to a residual $\varepsilon=6\cdot 10^{-6}$ and all non-intrusive methods are converged to a residual of $\varepsilon=1\cdot 10^{-6}$.}
	\label{fig:L2ErrorSolution2D}
\end{figure}
The error during the iteration has been recorded in Figure~\ref{fig:L2ErrorSolution2D}. To determine the error, we used stochastic-Collocation with a $50$ by $50$ Gauss-Legendre quadrature rule. As in the one-dimensional NACA testcase, the acceleration techniques lead to a heavily reduced runtime of the IPM method. Furthermore, the error obtained with the intrusive methods requires a total degree of $M = 9$, i.e. $N = 55$ moments to reach the error level of SC with 17 quadrature points per dimension i.e. $17^2 = 289$ collocation points. Note that this effect has also been observed in the one-dimensional case, however now for multi-dimensional problems, the reduced number of unknowns required for intrusive methods to obtain a certain error becomes more apparent. This shows one promising characteristic of intrusive techniques and points to their applicability for higher dimensional problems. In contrast to before, the SG error level is smaller than the IPM error for both, the expectation value and variance. Furthermore, when comparing IPM with and without One-Shot, one can observe that the effect of this acceleration strategy weighs in more heavily than it did for the one-dimensional case. This behavior is expected, since every iterate of the optimization problem becomes more expensive when the dimension is increased. This means, using a method with less iterations for every computation of the dual variables heavily reduces computational costs, which is why we consider the One-Shot strategy to be an effective approach, especially for problems with high uncertain dimension.
\begin{figure}[h!]
\centering
	\begin{subfigure}{0.329\linewidth}
		\centering
		\includegraphics[width=\linewidth]{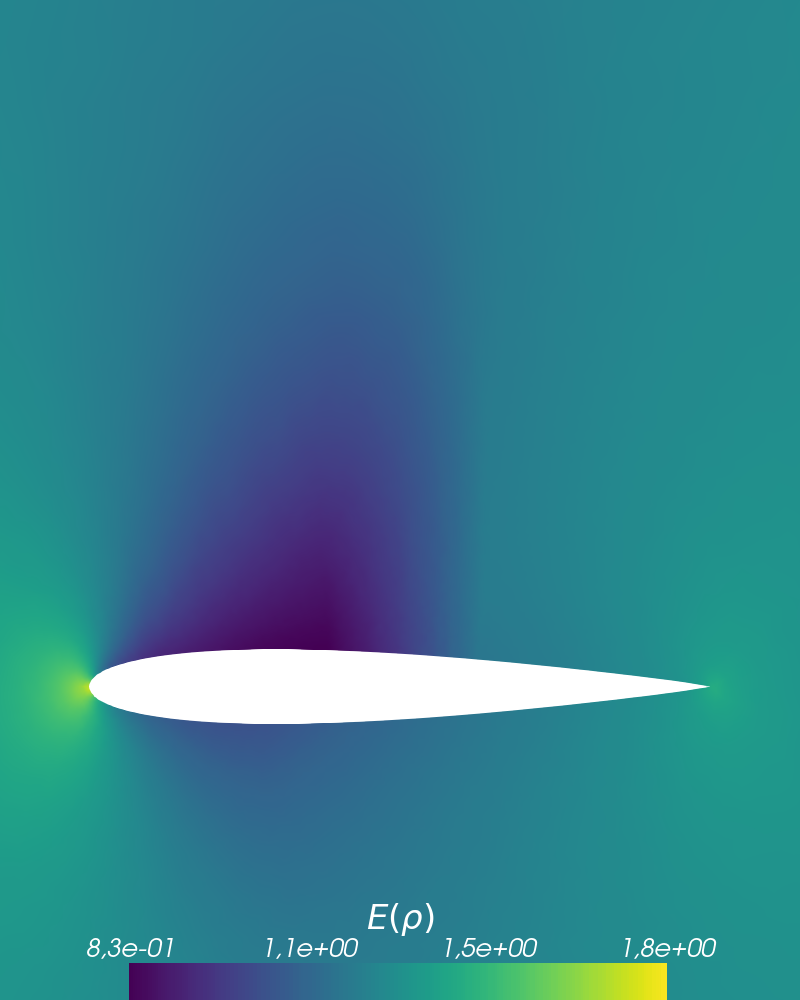}
		\caption{}
		\label{fig:referenceSolutionsub2DE}
	\end{subfigure}
	\hfill
	\begin{subfigure}{0.329\linewidth}
		\centering
		\includegraphics[width=\linewidth]{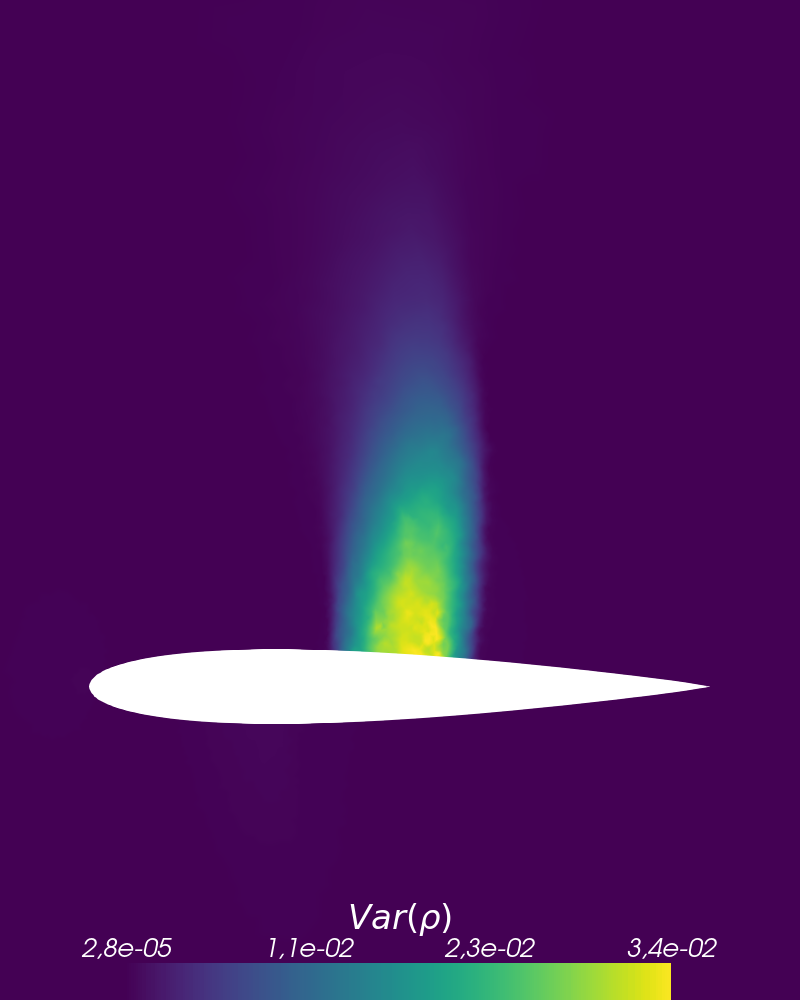}
		\caption{}
		\label{fig:referenceSolutionsub2DVar}
	\end{subfigure}
	\hfill
	\begin{subfigure}{0.329\linewidth}
		\centering
		\includegraphics[width=\linewidth]{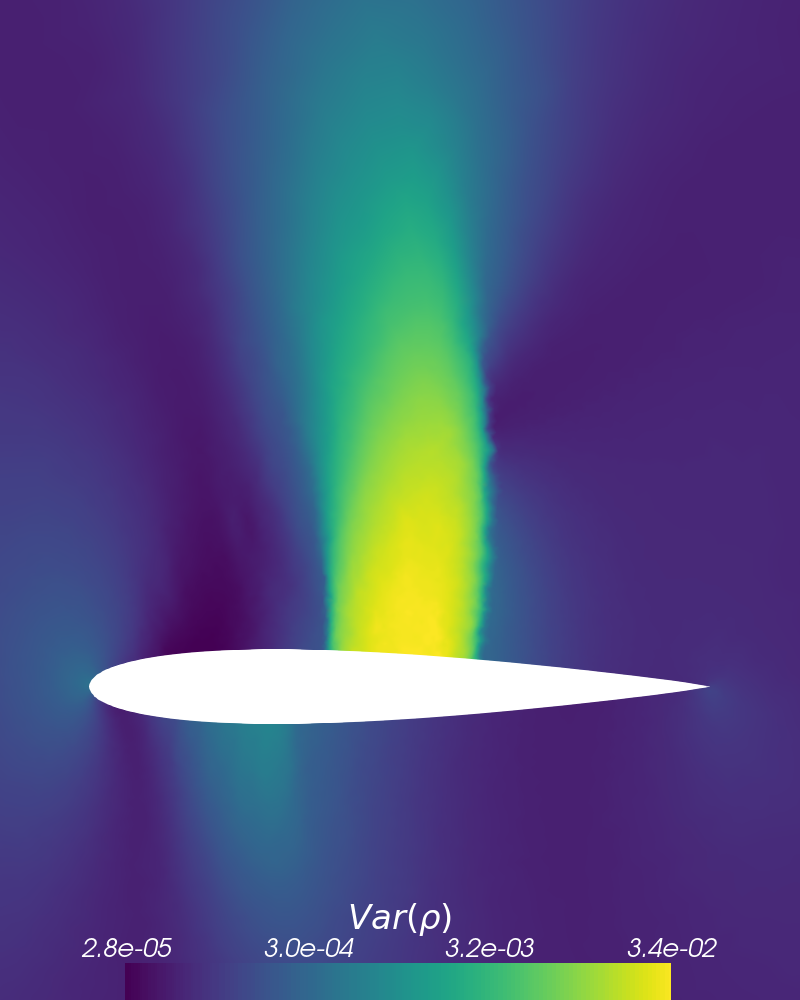}
		\caption{}
		\label{fig:referenceSolutionsub2DVarLog}
	\end{subfigure}\\
	\begin{subfigure}{0.329\linewidth}
		\centering
		\includegraphics[width=\linewidth]{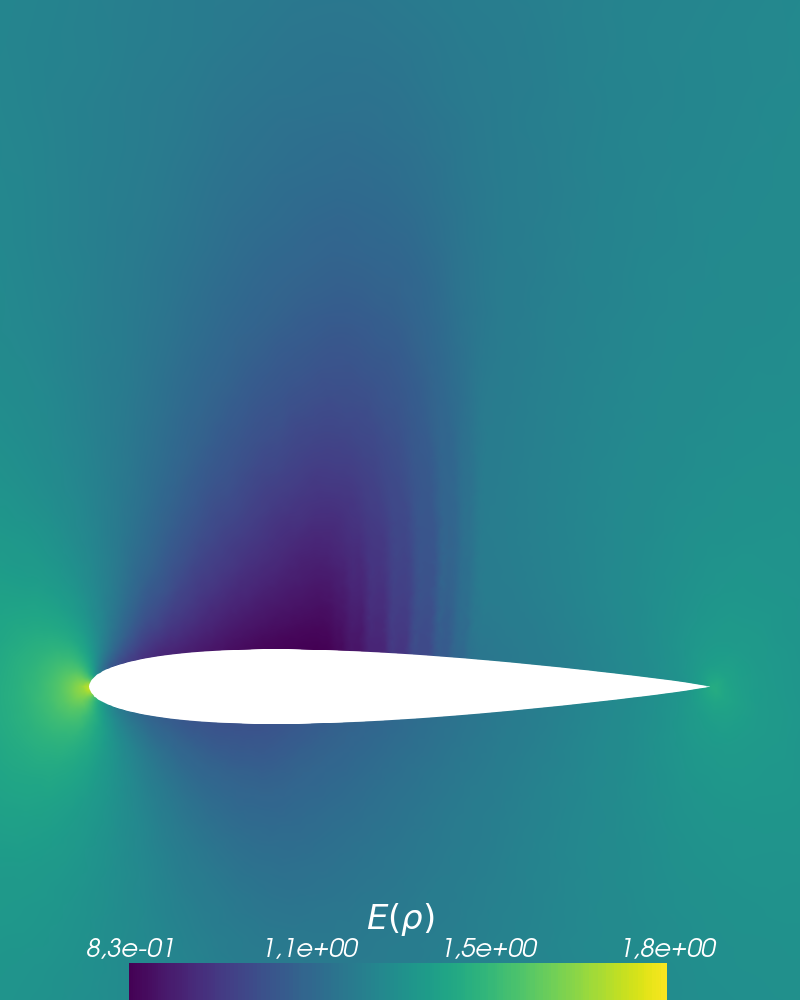}
		\caption{}
		\label{fig:reados2DE}
	\end{subfigure}%
	\hfill
	\begin{subfigure}{0.329\linewidth}
		\centering
		\includegraphics[width=\linewidth]{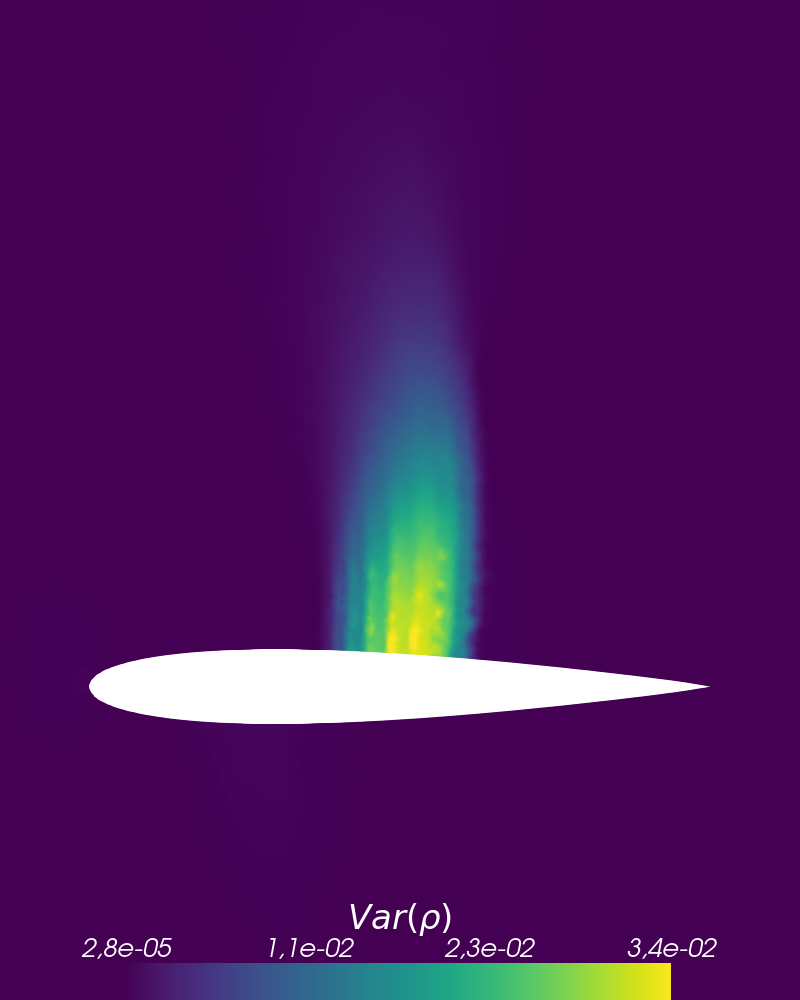}
		\caption{}
		\label{fig:reados2DVar}
	\end{subfigure}%
	\hfill
	\begin{subfigure}{0.329\linewidth}
		\centering
		\includegraphics[width=\linewidth]{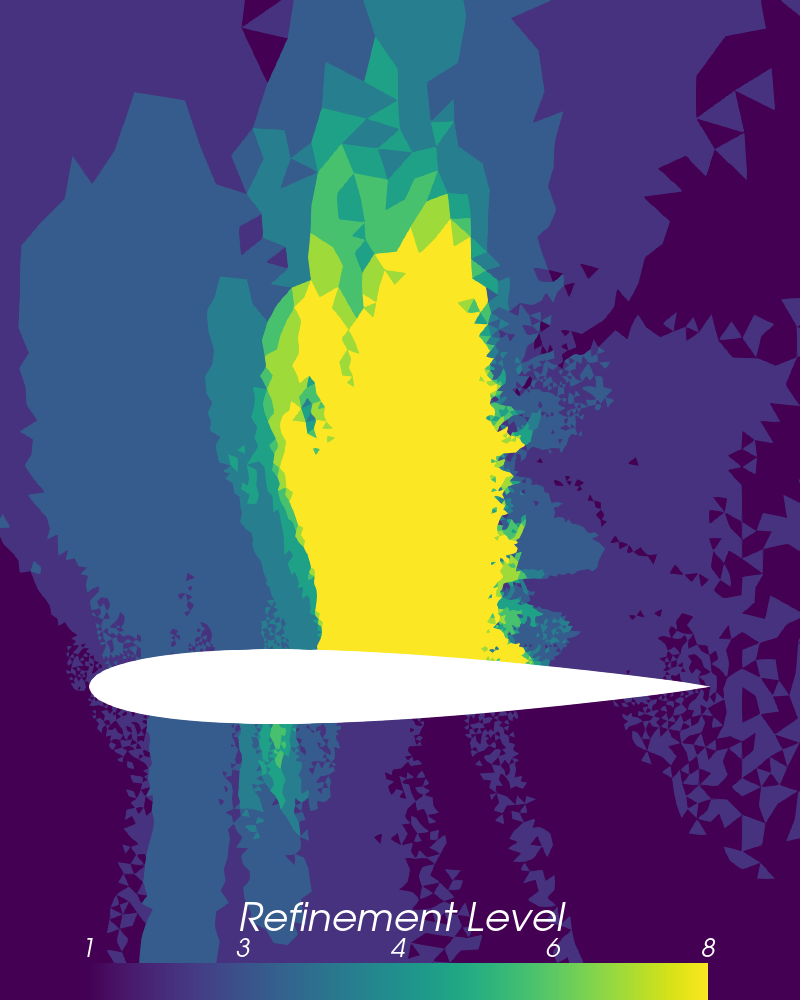}
		\caption{}
		\label{fig:readosIPMEVar2DRefinement}
	\end{subfigure}
	\caption{Reference solution (\subref{fig:referenceSolutionsub2DE}) E$[\rho]$ and (\subref{fig:referenceSolutionsub2DVar}) Var$[\rho]$ and with logarithmic scaled (\subref{fig:referenceSolutionsub2DVarLog}) Var$[\rho]$. readosIPM$_{1-9}$ solution for (\subref{fig:reados2DE}) E$[\rho]$ and (\subref{fig:reados2DVar}) Var$[\rho]$, with the resulting (\subref{fig:readosIPMEVar2DRefinement}) refinement levels.}
	\label{fig:refAndreadosIPMEVar2D}
\end{figure}
When looking at the computed IPM expectation value and variance (see Figure~\ref{fig:reados2DE} and \ref{fig:reados2DVar}) and comparing the results against the reference solution in Figure~\ref{fig:referenceSolutionsub2DE} and \ref{fig:referenceSolutionsub2DVar}, one can observe that since the uncertainties have a bigger effect on the result than in the one-dimensional case, the IPM solution still shows a step-like profile. However, we are able to capture the main features of the reference solution. The increased effect of the uncertainty can again be seen in Figure~\ref{fig:readosIPMEVar2DRefinement}, where a refinement level of $8$, i.e. a truncation order of $9$ is chosen on a larger portion than for the one-dimensional case. Visualizing the variance on a logarithmic scale in Figure~\ref{fig:referenceSolutionsub2DVarLog} further shows that the applied refinement indicators work well and enforce the usage of higher order moments in areas of high variance.

\subsection{2D Euler equations with a three dimensional uncertainty}
For the last numerical study we again use the Euler equations, but this time not in the context of the NACA airfoil, but rather with a bend Sod shock tube experiment. The prescribed initial condition for this testcase describes a gas at rest, but with a discontinuity, in density as well as energy in the upper part of the tube. The density $\rho_u = p/(R \cdot T)$ in the upper part is set to $\rho_u = 1.289$ with $p = 101\;325$ Pa, a temperature of $273.15$ K and the specific gas constant for dry air $R=287.87$. The energy $\rho e_u$ is set to exactly $1.0$. For the initial conditions in the lower part we set $\rho_l = 0.5 \cdot \rho_u$ and $\rho e_l = 0.3$. The heat capacity ratio $\gamma$ is set to $1.4$ as in the previous studies. We again augment the deterministic case by inflicting uncertainties and again increase the number uncertainties to three. The applied boundary conditions are Euler slip conditions $\bm{v}^T\bm{n}=0$ for the walls and Dirichlet type boundary conditions at the ends of the tube which set to the deterministic initial conditions. The first two uncertainties are the density and energy of the lower part of the shock tube. Here we set $\rho_l \sim U(1.189,1389)$ and $\rho e_l \sim U(0.2,0.4)$. The third uncertainty enters through the position of the shock itself and thus $y_\text{shock} \sim U(1.0, 1.2)$. Note, that in contrast to the upper testcases, here we are not interested in the steady state of the system, but rather the of time evolution of the expectancy and variance. The computational mesh (see Figure \ref{fig:referenceSolutionsPipeMesh}) is an unstructured mesh of $25\;458$ triangular cells, where the cells are all similar in size, i.e. there are not refined regions as in the NACA testcase. The simulations were run until a time of $2.0$s. For the refinement barries values of $\delta_{-} = 2\cdot 10^{-2}$ and $\delta_{+} = 4\cdot 10^{-3}$ were set for all adaptive simulations. The corresponding results can be seen in Figure \ref{fig:solution3D}. The reference solution in Figure \ref{fig:referenceSolutionsPipe1} and \ref{fig:referenceSolutionsPipe2} was computed using SC with 50 quadrature points in each stochastic dimension, yielding a total of $50^3=125\;000$ quadrature points.\\
For this testcase we want to compare the results for sparse grids with respect to tensorized grids for higher stochastical dimensions, where both quadrature sets are based on Clenshaw-Curtis nodes. As we made used of adaptivity, the corresponding moment orders and used quadrature points can be seen in Table \ref{tab:orders3D}.
\begin{table}
	\centering
	\begin{tabular}{l|cccc}
		Sparse grids&&&&\\\hline
		\quad Moment order&1&2&-&-\\
		\quad Number of quadrature points&25&441&-&-\\
		Tensorized grids & & & & \\\hline
		\quad Moment order&1&2&3&4\\
		\quad Number of quadrature points&27&125&125&729\\
	\end{tabular}
	\caption{Order of moments and corresponding number of quadrature nodes used.}
	\label{tab:orders3D}
\end{table}
The following properties emerge:
\begin{itemize}
	\item It was not possible to obtain a result for sparse grids with moments of order 3 or higher. Any level up to level 11 ($72\;705$ quadrature nodes) of the used Clenshaw-Curtis sparse grids, resulted in an ill conditioned Hessian matrix of the dual problem and thereby to a failure of the method.
	\item Sparse grids required a significantly higher number of quadrature points for the same order of moments in comparison to tensorized grids.
\end{itemize}
Even though the solutions in Figure \ref{fig:solution3D} show the same characteristics and sparse grids are generally well suited and often used in combination with the SC method, our tests reveal that, at least for the Clenshaw-Curtis based sparse grids, these quadrature rules are not well suited to be used in combination with IPM as they require significantly more quadrature points and are only able to generate results for low orders of moments. Tensorized quadrature on the other hand performs similarly well as in the previous testcases and manages to yield a closer solution for the expectation as well as the variance due to the usage of higher moments. Tensorized quadrature rules should thus be preferred for lower stochastical dimensions. 
\begin{figure}[H]
	\centering
	\begin{subfigure}{0.31\linewidth}
		\centering
		\includegraphics[width=\linewidth]{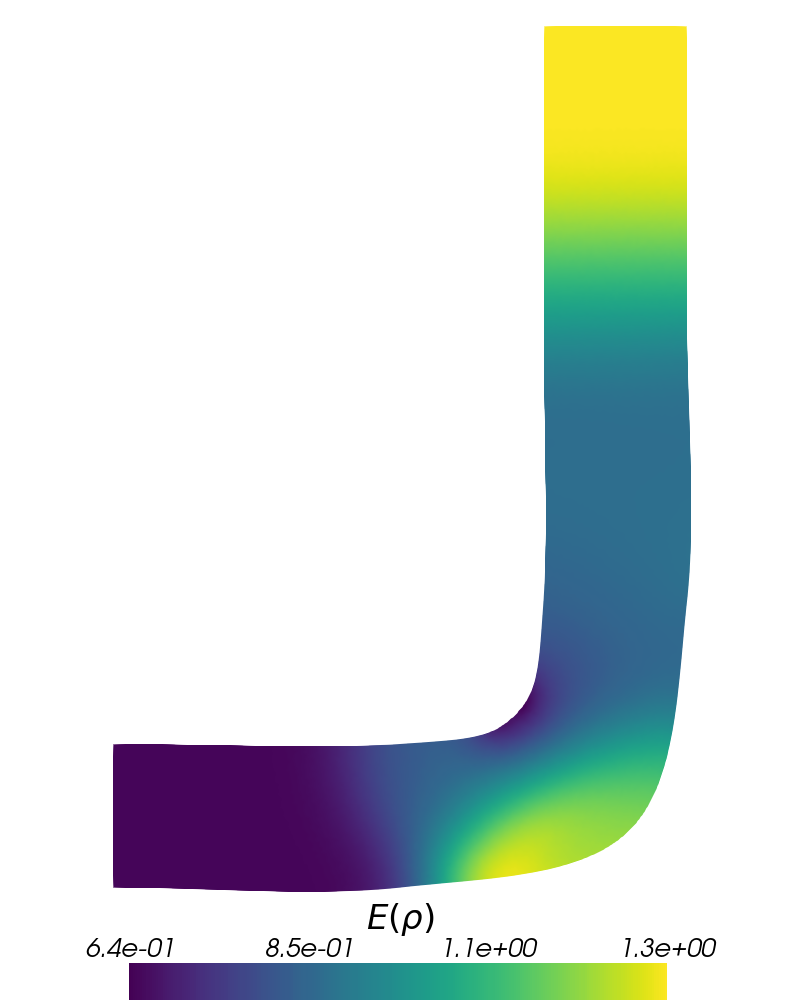}
		\caption{}
		\label{fig:referenceSolutionsPipe1}
	\end{subfigure}%
	\hfill
	\begin{subfigure}{0.31\linewidth}
		\centering
		\includegraphics[width=\linewidth]{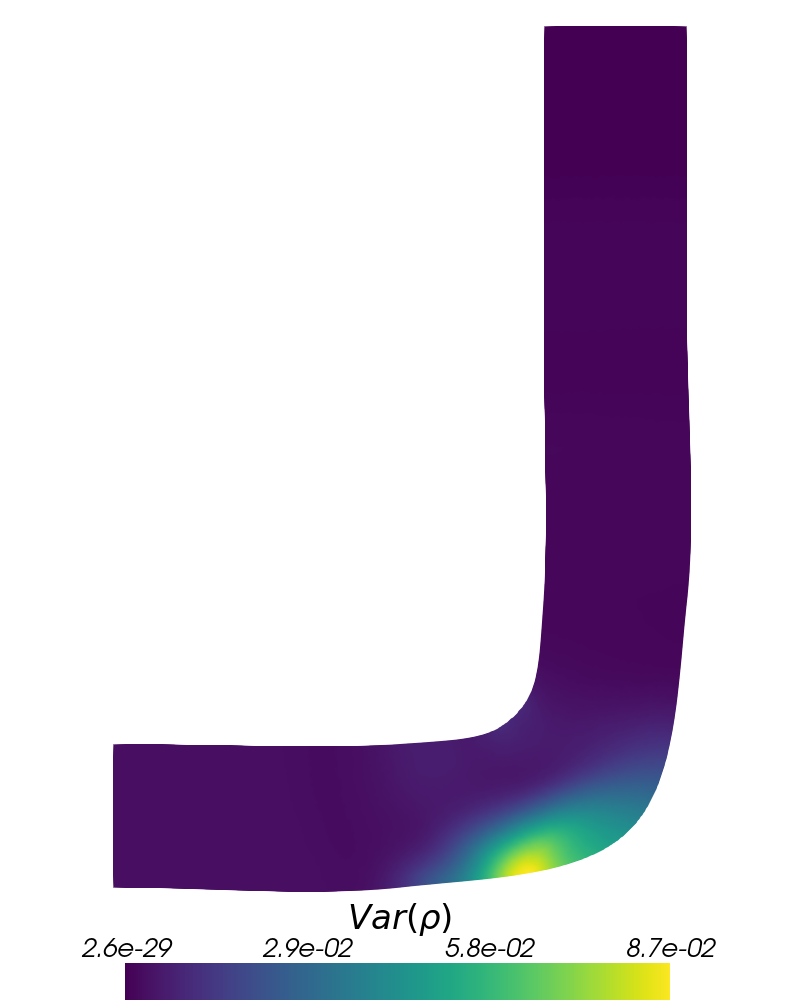}
		\caption{}
		\label{fig:referenceSolutionsPipe2}
	\end{subfigure}
	\hfill
	\begin{subfigure}{0.31\linewidth}
		\centering
		\includegraphics[width=\linewidth]{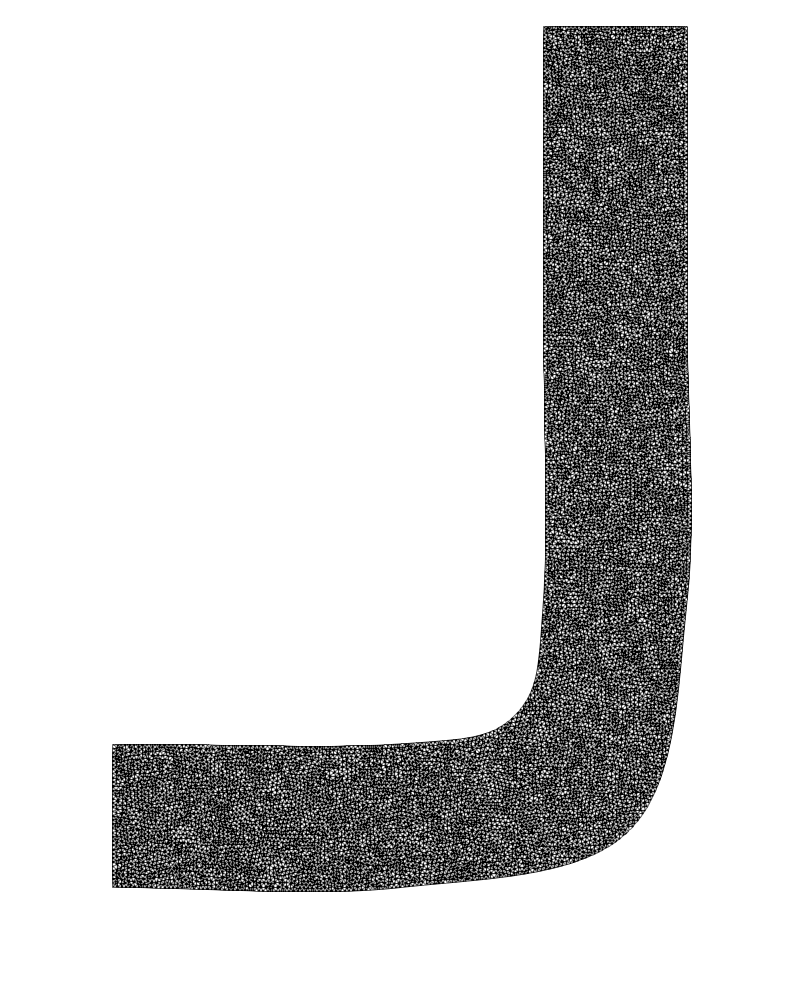}
		\caption{}
		\label{fig:referenceSolutionsPipeMesh}
	\end{subfigure}\\
	\begin{subfigure}{0.31\linewidth}
		\centering
		\includegraphics[width=\linewidth]{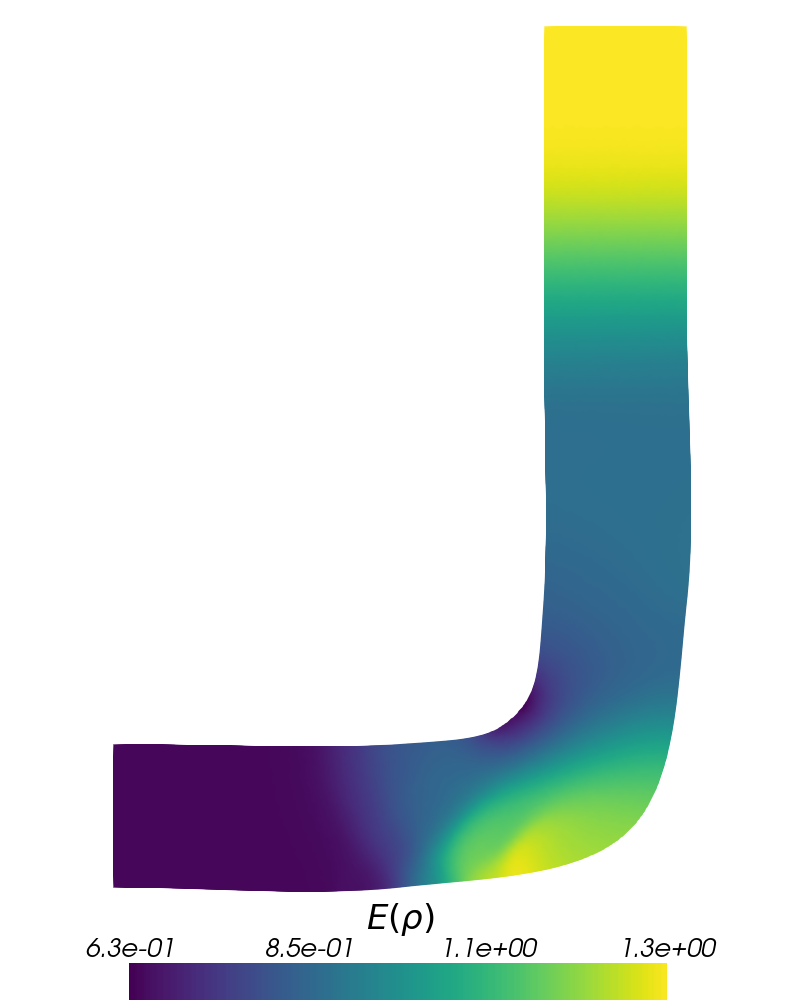}
		\caption{}
		\label{fig:adIPMSolutionsPipeSGE}
	\end{subfigure}%
	\hfill
	\begin{subfigure}{0.31\linewidth}
		\centering
		\includegraphics[width=\linewidth]{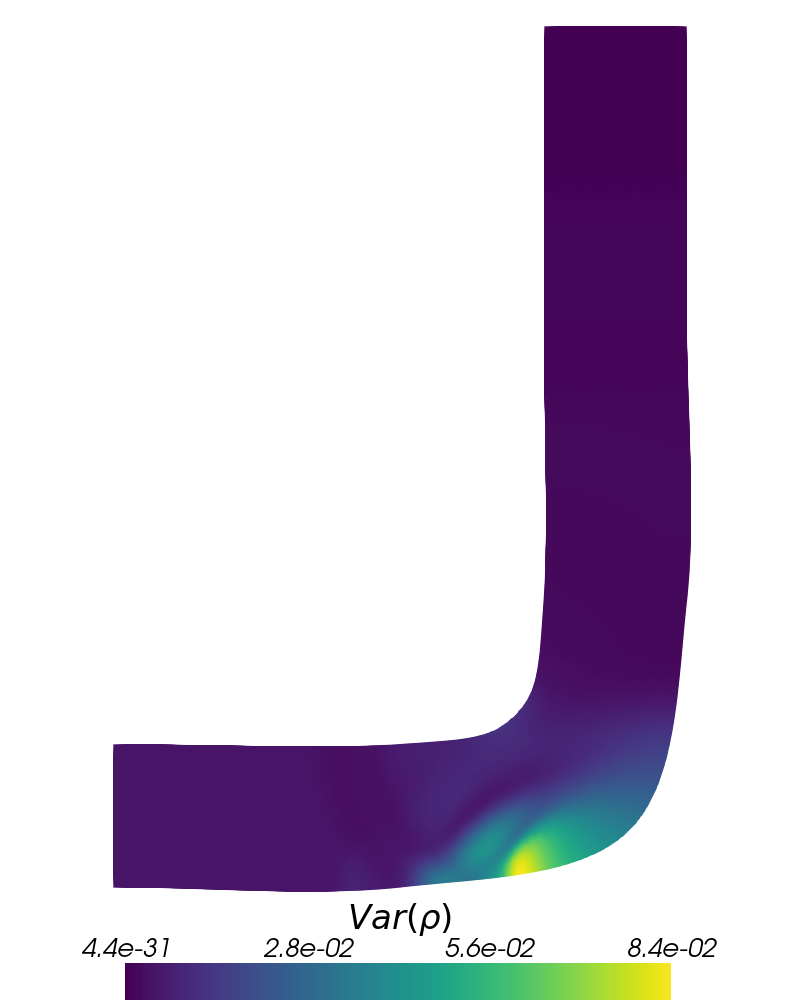}
		\caption{}
		\label{fig:adIPMSolutionsPipeSGVar}
	\end{subfigure}
	\hfill
	\begin{subfigure}{0.31\linewidth}
		\centering
		\includegraphics[width=\linewidth]{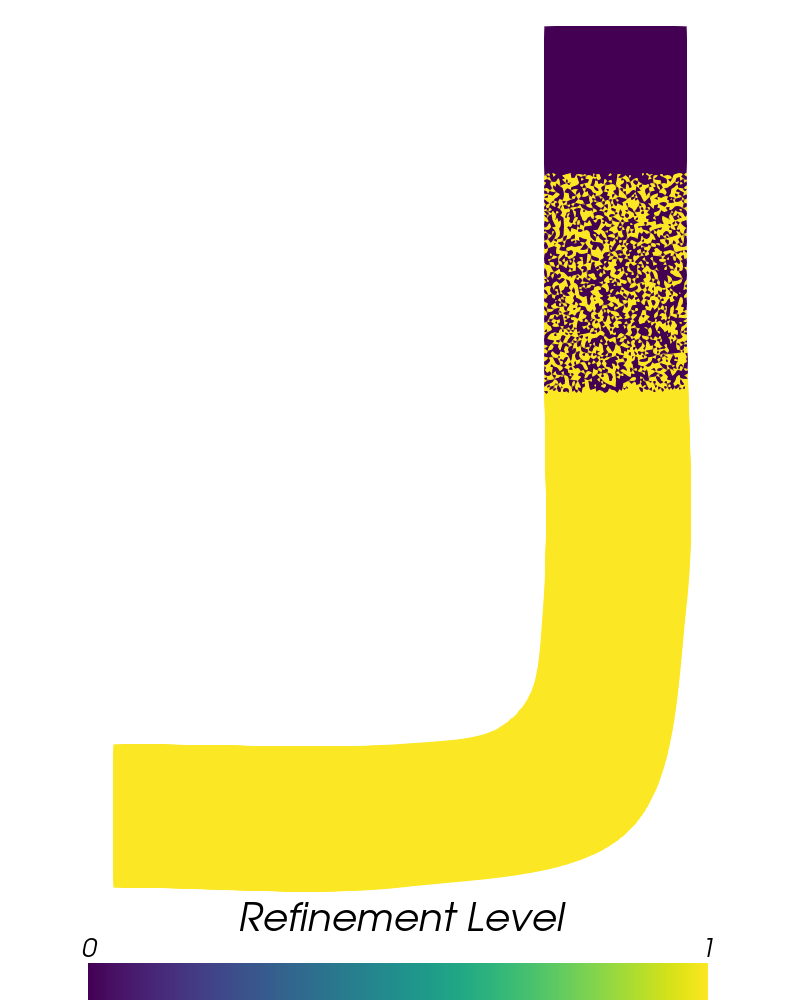}
		\caption{}
		\label{fig:adIPMSolutionsPipeSGRI}
	\end{subfigure}\\
	\begin{subfigure}{0.31\linewidth}
		\centering
		\includegraphics[width=\linewidth]{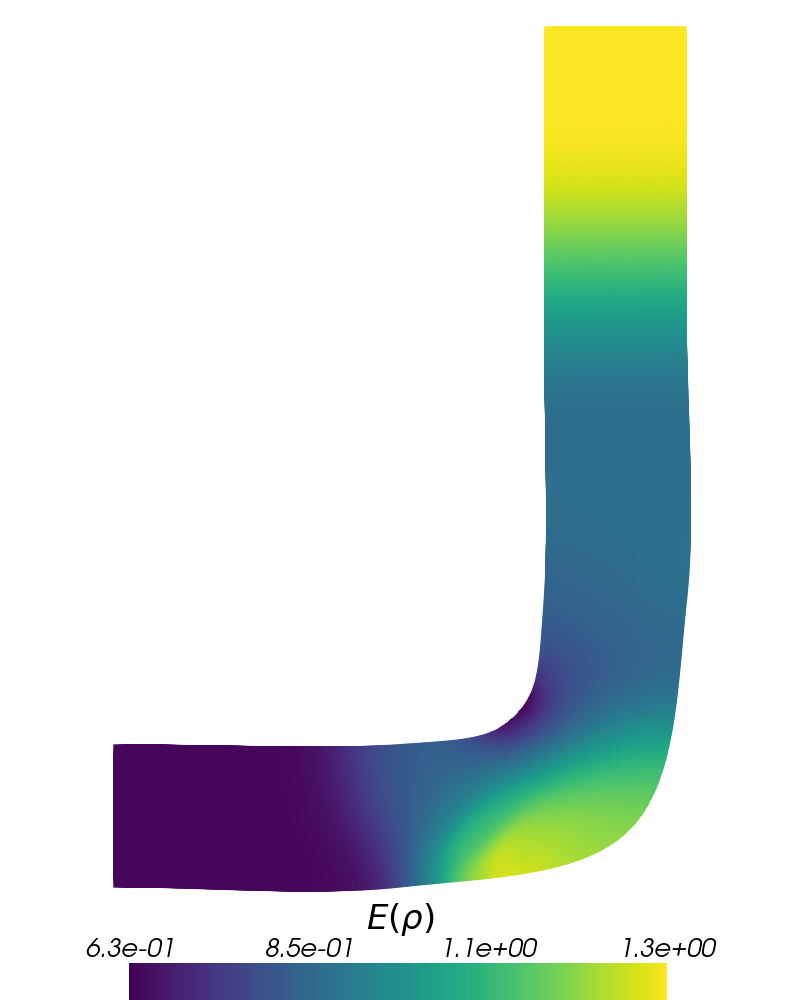}
		\caption{}
		\label{fig:adIPMSolutionsPipeTGE}
	\end{subfigure}%
	\hfill
	\begin{subfigure}{0.31\linewidth}
		\centering
		\includegraphics[width=\linewidth]{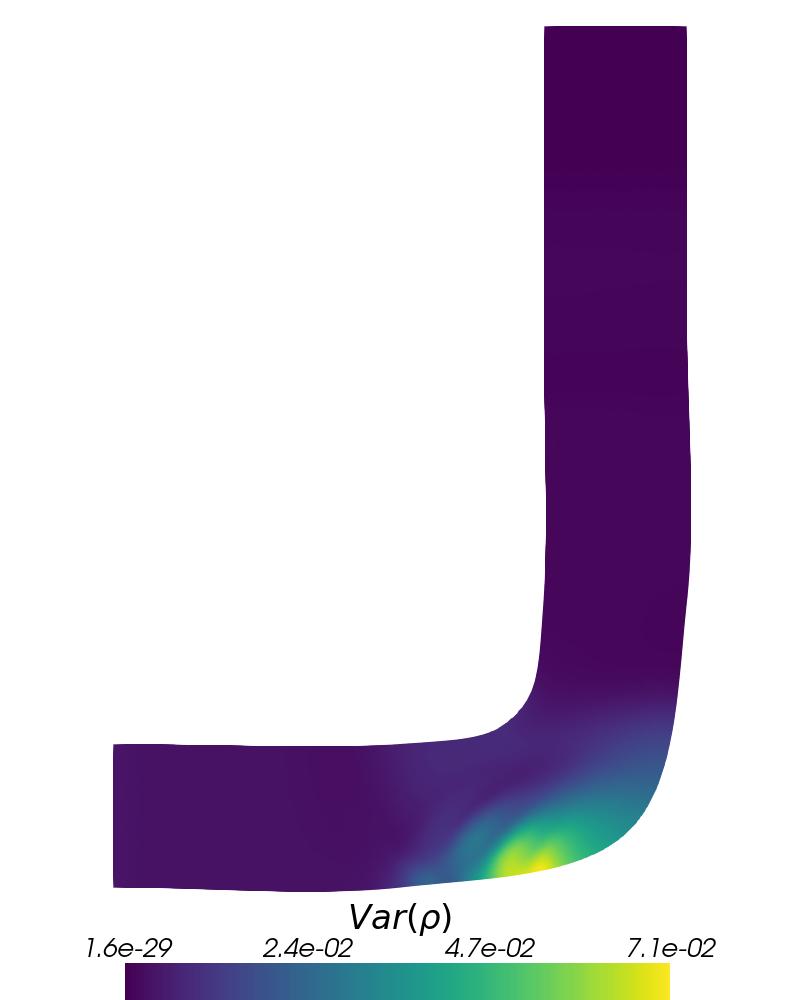}
		\caption{}
		\label{fig:adIPMSolutionsPipeTGVar}
	\end{subfigure}
	\hfill
	\begin{subfigure}{0.31\linewidth}
		\centering
		\includegraphics[width=\linewidth]{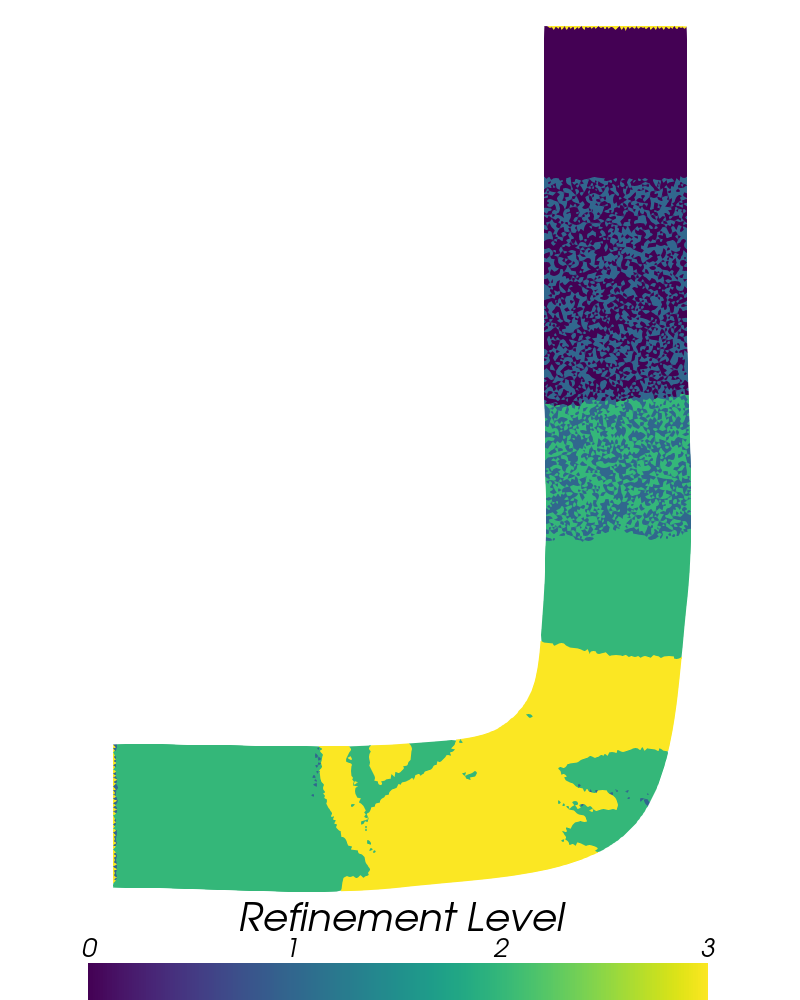}
		\caption{}
		\label{fig:adIPMSolutionsPipeTGRI}
	\end{subfigure}
	\caption{Reference solution E$[\rho]$ and Var$[\rho]$ and used computational mesh (top row) and corresponding adaptive IPM solutions for sparse (middle row) and tensorized grids (bottom row) with the used refinement levels at the last time step}
	\label{fig:solution3D}
\end{figure}
\section{Summary and outlook}
\label{sec:summary_outlook}
In this work, we proposed acceleration techniques, which are applicable to the intrusive nature of IPM and SG: We use a One-Shot technique to iterate moments and their corresponding dual variables to their steady state simultaneously. By not fully converging the dual iteration, we can reduce computational costs. Additionally, we make use of adaptivity and propose to keep a low maximal truncation order for most of the iteration to the steady state solution. Since complicated structures in the uncertain dimension only appear on a small portion of the spacial mesh, we are able to heavily reduce computational costs. The effects of the proposed techniques have been demonstrated by comparing results obtained with IPM as well as SG against SC. In our test cases, the intrusive methods yield the same error level as SC for a reduced runtime, especially since intrusive methods require less unknowns to achieve a certain accuracy due to aliasing errors. In higher-dimensional problems, this effect is amplified since the number of unknowns to achieve a certain total degree is asymptotically smaller than the number of quadrature points in a tensorized or sparse grid (\eqref{eq:numberBasisFcts} instead of $O(M(\log_2(M)^{p-1}))$). Furthermore, we could observe that the required residual at which the solution of intrusive methods reaches a steady state is smaller than for SC. Additionally, the ability to adaptively change the truncation order helps intrusive methods to compete with SC in terms of computational runtime.

In future work, we aim at further accelerating the IPM method by using non-exact Hessian approximations. Similar to the One-Shot idea of not fully converging the dual problem, it seems to be plausible to not spend too much time on computing the Hessian when the moments are not close to a steady state. Hessian approximations that can be interesting are BFGS and sparse BFGS \cite[Chapter~6.1]{nocedal2006numerical}, which construct the Hessian from previously computed gradients. Note that this strategy will increase the used memory, since old Hessians or gradients from a certain number of old time steps need to be saved in every spacial cell.
Even though the non-intrusive nature of Stochastic Collocation or Monte Carlo methods allows an easy implementation, as shown in this work, it can be helpful to intrusively modify the code in order to fully exploit all acceleration potentials. Synchronizing the time updates of the solution at different quadrature points yields an increased control over the solution during the computation, which can for example be uses to employ adaptive methods. In this case one can switch to a fine quadrature level in a certain spatial cell by for example computing moments with the given coarse set of collocation points. From these moments one can compute an IPM reconstruction, which one can then evaluate at a finer quadrature set. Another example of breaking up the non-intrusive nature of Monte Carlo methods can be found in \cite{poette2019gpc}, where the generation of random samples is combined with the sampling after collisions to increase efficiency. Furthermore, we aim at applying the proposed acceleration techniques to SG methods with hyperbolicity limiters \cite{wu2017stochastic,schlachter2018hyperbolicity}.

\section*{Acknowledgment} \noindent
This work was funded by the Deutsche Forschungsgemeinschaft (DFG, German Research Foundation) – FR 2841/6-1.

\newpage
\bibliographystyle{unsrt}  
\bibliography{references}
\appendix
\section{Costs of evaluating the numerical flux}
\label{app:costNumFlux}
In the following, we briefly discuss the number of operations needed when precomputing integrals versus the use of a kinetic flux for Burgers' equation. The stochastic Burgers' equation reads
\begin{align*}
\partial_t u + \partial_x \frac{u^2}{2} &= 0,\\
u(t=0,x,\xi) &= u_{IC}(x,\xi).
\end{align*}
The scalar random variable $\xi$ is uniformly distributed in the interval $[-1,1]$, hence the gPC basis functions $\bm\varphi=(\varphi_0,\cdots,\varphi_M)^T$ are the Legendre polynomials. Choosing the SG ansatz \eqref{eq:SGClosure} and testing with the gPC basis functions yields the SG moment system
\begin{align*}
\partial_t \hat u_i + \partial_x \frac12\sum_{n,m = 0}^M \hat u_n \hat u_m \langle \varphi_n\varphi_m\varphi_i \rangle = 0.
\end{align*}
Defining the matrices $\bm C_i := \langle \bm\varphi\bm\varphi^T\varphi_i\rangle\in\mathbb{R}^{N\times N}$ gives
\begin{align*}
\partial_t \bm{\hat u} + \partial_x \bm F(\bm{\hat u}) = \bm 0
\end{align*}
with $F_i(\bm{\hat u}) = \frac12\bm{\hat u}^T\bm C_i\bm{\hat u}$. Note that $\bm{C}_i$ can be computed analytically, hence choosing a Lax-Friedrichs flux
\begin{align}\label{eq:numFluxAnalytic}
G_i^{(LF)}(\bm{\hat u}_{\ell},\bm{\hat u}_{r}) =\frac{1}{4}\left(\bm{\hat u}_{\ell}^T \bm{C}_i \bm{\hat u}_{\ell}+\bm{\hat u}_{r}^T \bm{C}_i \bm{\hat u}_{r}\right) - \frac{\Delta x}{2\Delta t}(\bm{\hat u}_{r}-\bm{\hat u}_{\ell})_i
\end{align}
requires no integral evaluations. Recall, that the numerical flux choice made in this work gives
\begin{align}\label{eq:numericalFluxIPMBurgers}
 \bm{G}(\bm{\hat u}_{\ell},\bm{\hat u}_{r}) = \sum_{k=1}^Q w_k g(\mathcal{U}(\bm{\hat u}_{\ell};\xi_k),\mathcal{U}(\bm{\hat u}_{r};\xi_k))\bm{\varphi}(\xi_k)f_{\Xi}(\xi_k),
\end{align}
where $\mathcal{U}$ is the SG ansatz \eqref{eq:SGClosure}. When the chosen deterministic flux $g$ is Lax-Friedrichs, the order of the polynomials inside the sum is $3M=3(N-1)$. Choosing a Gauss-Lobatto quadrature rule, $Q = \frac32 N -1$ quadrature points suffice for an exact computation of the numerical flux. Indeed, with this choice of quadrature points, the numerical fluxes \eqref{eq:numFluxAnalytic} and \eqref{eq:numericalFluxIPMBurgers} are equivalent. 
Counting the number of operations, one observes that our choice of the numerical flux \eqref{eq:numericalFluxIPMBurgers} uses $O(N^2)$ operations whereas \eqref{eq:numFluxAnalytic} requires $O(N^3)$ operations: When computing and storing the values in a matrix $\bm A\in\mathbb{R}^{Q\times N}$ with entries $a_{ki} = \varphi_i(\xi_k)$ before running the program, the numerical flux \eqref{eq:numericalFluxIPMBurgers} can be split into two parts. First, we determine the SG solution at all quadrature points, i.e. we compute $\bm{u}^{(\ell)} := \bm A \bm{\hat u}_{\ell}$ and $\bm{u}^{(r)} := \bm A \bm{\hat u}_{r}$ which requires $O(N\cdot Q)$ operations. These solution values are then used to compute the numerical flux
\begin{align*}
G_i(\bm{\hat u}_{\ell},\bm{\hat u}_{r}) &= \sum_{k=1}^Q w_k g(u^{(\ell)}_k,u^{(r)}_k)a_{ki}f_{\Xi}(\xi_k),
\end{align*}
which again requires $O(N\cdot Q)$ operations, i.e. the costs are $O(N^2)$. The evaluation of \eqref{eq:numFluxAnalytic} however requires $O(N^3)$ operations.

\section{IPM for the 2D Euler equations}
\label{app:IPM2DEuler}

In the following, we provide details on the implementation of IPM for the 2D Euler equations. For ease of presentation, we denote the momentum by $m_1 := \rho v_1$ and $m_2:=\rho v_2$ and the energy by $E:=\rho e$. Then, the vector of conserved variables is $\bm u = (\rho,m_1,m_2,E)^T$. The entropy used is
\begin{align*}
s(\bm u) = -\rho \ln \left(\rho^{-\gamma} \left(E - \frac{m_1^2 + m_2^2}{2
\rho}\right)\right).
\end{align*}
Now the gradient of the entropy $\nabla_{\bm u} s$ has the components
\begin{align*}
\frac{\partial s}{\partial \rho} &= -\ln \left(\rho^{-\gamma}\left(E-\frac{m_1^2+m_2^2}{2 \rho }\right)\right)+\frac{m_1^2+m_2^2}{-2 \rho  E+m_1^2+m_2^2}+\gamma, \\
\frac{\partial s}{\partial m_i} &= -\frac{2\rho  m_i}{-2 \rho  E+m_1^2+m_2^2}, \\
\frac{\partial s}{\partial E} &=-\frac1\rho\left(E-\frac{m_1^2+m_2^2}{2 \rho }\right).
\end{align*}
To compute $\bm u_s(\bm\Lambda) = (\nabla_{\bm u}s)^{-1}(\bm \Lambda)$, we set $\bm \Lambda = \nabla_{\bm u}s(\bm u)$ and rearrange with respect to $\bm u$. Let us define
\begin{align*}
\alpha(\bm\Lambda) := \text{exp}\left(\frac{ \Lambda_2^2 + \Lambda_3^2 - 2\Lambda_1  \Lambda_4 - 2 \Lambda_4  \gamma }{ 2 \Lambda_4(1-\gamma) } \right) \cdot (-\Lambda_4)^{\frac{1}{1-\gamma}}
\end{align*}
Then the solution ansatz $\bm u_s$ is given by
\begin{align*}
\rho(\bm\Lambda) &= \alpha(\bm\Lambda),\qquad m_1(\bm\Lambda) = -\frac{\Lambda_2 \alpha(\bm{\Lambda})}{\Lambda_4},\qquad m_2(\bm\Lambda) = -\frac{\Lambda_3 \alpha(\bm{\Lambda})}{\Lambda_4}, \\
E(\bm\Lambda) &= -\frac{  \alpha(\bm{\Lambda}) ( -\Lambda_2^2 - \Lambda_3^2 + 2\Lambda_4 ) }{ 2 \Lambda_4^2}.
\end{align*}
\end{document}